\documentclass[12pt]{amsart}
\usepackage[margin=2.2cm]{geometry}

\usepackage[OT2,T1]{fontenc}

\usepackage{amsthm}
\usepackage{verbatim}
\usepackage[normalem]{ulem}
\usepackage{hyperref}
\usepackage{amssymb,mathrsfs,mathcomp,amsfonts,longtable, hhline, textcomp, upgreek}
\usepackage[matrix,arrow,curve]{xy}
\usepackage{tabularx,multirow}
\usepackage{float}
\usepackage{tikz}
\usepackage{graphicx}

\newcommand{\mumu}{{\boldsymbol{\mu}}}
\newcommand{\B}{{\mathbf{B}}}
\newcommand{\CC}{\mathbb{C}}

\newcommand{\QQ}{\mathbb{Q}}

\newcommand{\ZZ}{\mathbb{Z}}
\newcommand{\PP}{\mathbb{P}}
\newcommand{\OOO}{{\mathscr{O}}} 
 
\newcommand{\MMM}{{\mathscr{M}}} 
\newcommand{\NNN}{{\mathscr{N}}}

\newcommand{\h}{\operatorname{h}}
\newcommand{\p}{\operatorname{p}}
\newcommand{\pa}{\operatorname{p}_{\mathrm{a}}}
\newcommand{\g}{\operatorname{g}}

\newcommand{\qq}{\mathbin{\sim_{\scriptscriptstyle{\mathbb{Q}}} } }
\newcommand{\qW}{\operatorname{q}_{\operatorname{W}}}
\newcommand{\qQ}{\operatorname{q_{\QQ}}}

\newcommand{\Sing}{\operatorname{Sing}}

\newcommand{\Bs}{\operatorname{Bs}}
\newcommand{\Pic}{\operatorname{Pic}}
\newcommand{\Cl}{\operatorname{Cl}}
\newcommand{\Clt}[1]{\operatorname{Cl}(#1)_{\mathrm {t}}}
\newcommand{\rk}{\operatorname{rk}}
\newcommand{\aw}{\operatorname{aw}}

\newcommand{\ct}{\operatorname{ct}}
\newcommand{\Exc}{\operatorname{Exc}}
\newcommand{\ind}{\mathrm{{r}}}
\newcommand{\lcm}{\operatorname{lcm}}
\newcommand{\mult}{\operatorname{mult}}

\newcommand{\type}[1]{$\mathrm{#1}$}
\newcommand{\typem}[1]{$\mathbf{#1}$}
\newcommand{\typec}[1]{$(\mathrm{#1})$}
\newcommand{\typeci}[2]{$(\mathrm{#1}_{#2})$}
\newcommand{\types}[2]{$\mathrm{#1}_{#2}$}
\newcommand{\typeA}[2]{$\mathrm{#1}/{#2}$}

\newcommand{\xref}[1]{{\rm~\ref{#1}}}
\makeatletter
\@addtoreset{equation}{section}
\makeatother

\theoremstyle{definition}

\swapnumbers
\theoremstyle{plain}
\newtheorem{theorem}[subsection]{Theorem}
\newtheorem{lemma}[subsection]{Lemma}

\newtheorem{proposition}[subsection]{Proposition}

\newtheorem{scorollary}[subsubsection]{Corollary}
\newtheorem*{claim*}{Claim}

\newtheorem{slemma}[subsubsection]{Lemma}
\newtheorem{sproposition}[subsubsection]{Proposition}

\newtheorem{sproposition-definition}[subsubsection]{Proposition-Definition}
\theoremstyle{definition}
\newtheorem{setup}[subsection]{Set-up}
\newtheorem{definition}[subsection]{Definition}

\newtheorem*{definition*}{Definition}

\newtheorem{example-remark}[subsection]{Remark-Example}
\newtheorem{subexample-remark}[subsubsection]{Remark-Example}

\newtheorem*{notation*}{Notation}
\newtheorem{example}[subsection]{Example}
\newtheorem{examples}[subsection]{Examples}

\newtheorem{remark}[subsection]{Remark}

\newtheorem{sremark}[subsubsection]{Remark}

\newcounter{NN}\numberwithin{NN}{section}%{theorem}
\renewcommand{\theNN}{\arabic{NN}${}^o$}
\def\nr{~\refstepcounter{NN}{\theNN}}%

\renewcommand{\theenumi}{\rm (\roman{enumi})}
\renewcommand{\labelenumi}{\rm (\roman{enumi})}

\begin{document}
\title{On the birational geometry of $\QQ$-Fano threefolds\\ of large Fano index, II} 

\address{ 
Steklov Mathematical Institute of Russian Academy of Sciences, Moscow, Russian Federation
} 
\email{prokhoro@mi-ras.ru}

\author{Yuri~Prokhorov}
\thanks{This work is supported by the Russian Science Foundation under grant no. 23-11-00033, 
\url{https://rscf.ru/project/23-11-00033/}}
% \thanks{\todayy}

\begin{abstract}
This paper is a sequel to \cite{P:QFano-rat1}.
We investigate the rationality problem for $\QQ$-Fano threefolds of Fano index $\ge 3$. 
\end{abstract}

\maketitle

\section{Introduction}
A $\QQ$-Fano variety is a projective variety 
with only terminal $\QQ$-factorial singularities such that 
$-K_X$ is ample and $\rk\Pic(X)=1$. These varieties are very important in birational geometry since they 
naturally appear as one of the outputs of the Minimal Model Program.
This paper is a continuation of the series of our works \cite{P:fano-conic}, \cite{P:2019:rat:Q-Fano}, \cite{P:QFano-rat1}, \cite{P:P11223}
where we discuss the birational geometry of $\QQ$-Fano threefolds of large Fano 
index.

The \textit{$\QQ$-Fano index} of a $\QQ$-Fano variety $X$ is the maximal integer 
$\qQ(X)$ that divides the
canonical class $K_X$ in the Weil divisor class group modulo torsion 
(see~\eqref{eq:deq-q}). 
It is known that in the three-dimensional case this invariant takes the value in 
the set $\{1,2,\dots,9,11,13,17,19\}$ (see \cite{Suzuki-2004} and 
\cite{P:2010:QFano}). We call 
a Weil divisor $A$ such that 
\[
-K_X\qq \qQ(X) A
\]
the \textit{fundamental 
divisor} 
and denote it by $A_X$.
There is another important set of invariants of a $\QQ$-Fano threefold:
\begin{equation*}
\label{def:pn}
\p_n(X):=\max \left\{ \h^0(X,\OOO_X(D)) \mid D\qq nA_X\right\}. 
\end{equation*} 
If the Weil divisor class group $\Cl(X)$ is torsion free, then the above 
definition becomes simpler:
\[
\p_n(X)=\h^0(X,\OOO_X(nA_X)).
\] 

\begin{theorem}[{\cite{P:2019:rat:Q-Fano}}, {\cite{P:QFano-rat1}}]
\label{thm0}
Let $X$ be a $\QQ$-Fano threefold with $\qQ(X)\ge 2$.
If one of the following conditions holds then $X$ is rational
\begin{enumerate}
\item
$\p_1(X)\ge4$,
\item
$\qQ(X)\ge 3$ and $\p_1(X)\ge3$,
\item
$\qQ(X)\ge 4$ and $\p_1(X)\ge2$,

\item
\label{thm0:5}
$\qQ(X)\ge 5$, $\p_2(X)\ge2$, and $A_X^3\neq 1/12$,

\item
$\qQ(X)\ge 6$ and $\p_3(X)\ge2$,
\item
$\qQ(X)\ge 8$.
\end{enumerate}
\end{theorem}

In this paper we improve these results and investigate birational properties of 
``extremal'' 
varieties. Below we present three distinguished $\QQ$-Fano weighted 
hypersurfaces.
According to \cite{Okada2019} \textit{very general} members of these families 
are 
not stably rational. However this ``very general'' condition is not explicit.
We are interested in detailed birational geometry of some of varieties of this 
kind.

\begin{example}
\label{ex:X6}
Let $X=X_{6}\subset\PP(1,1,2,2,3)$ be a hypersurface of degree $6$.
A general variety of this type is  a $\QQ$-Fano 
threefold with
\begin{equation*}
%\label{eq:ex:q=5}
\qQ(X)=3,\quad A_X^3=1/2,\quad\B(X)=(2^3),\quad
\dim|A_X|=1,\quad \dim|2A_X|=4.
\end{equation*}
According to \cite{P:P11223}
$X$ is not rational if any non-Gorenstein singularity is moderate 
\cite{Kawamata:Moderate}, any Gorenstein
singularity is either a node or cusp, and the number of these Gorenstein
singularities is at most $4$.
\end{example}

\begin{example}
\label{ex:X10}
Let $X=X_{10}\subset\PP(1,2,3,4,5)$ be a hypersurface of degree $10$.
A general variety of this type is  a $\QQ$-Fano 
threefold with
\begin{equation*}
%\label{eq:ex:q=5}
\qQ(X)=5,\qquad A_X^3=1/12,\quad\B(X)=(2^2, 3, 4),\quad
\dim|A_X|=0,\quad \dim|2A_X|=1.
\end{equation*}
Thus $X$ has a unique point of index~$3$ that is a cyclic quotient
and a unique point of index~$4$ that is either a cyclic quotient or a
singularity of type~\typeA{cAx}{4}. 
\end{example}

\begin{example}
\label{ex:X14}
Let $X=X_{14}\subset\PP(2,3,4,5,7)$ be a hypersurface of 
degree $14$.
A general variety of this type is  a $\QQ$-Fano 
threefold with
\begin{equation*}
%\label{eq:ex:q=5}
\qQ(X)=7,\qquad A_X^3=1/60,\qquad\B(X)=(2^3, 3, 4, 5).
\end{equation*}
Thus $X$ has a unique point of index~$5$ that is a cyclic quotient,
a unique point of index~$3$ that is also a cyclic quotient, 
and a unique point of index~$4$ that is either a cyclic quotient or a
singularity of type~\typeA{cAx}{4}. 
\end{example}

The main results of this paper are as follows. 

\begin{theorem}
\label{thm:main}
Let $X$ be a $\QQ$-Fano threefold with $\qQ(X)\ge 2$.
\begin{enumerate}

\item
\label{thm:main6}
If $\qQ(X)=6$, then $X$ is rational.

\item
\label{thm:main7}
If $\qQ(X)=7$ and $X$ is not rational, then $\p_1(X)=0$ and $X$ is birationally 
equivalent to 
a hypersurface as in Example~\xref{ex:X10}.

\item 
\label{thm:main:5}
Assume that $\qQ(X)=5$, $X$ is not rational and at least one of the following 
conditions holds:
\begin{enumerate}
\item 
\label{thm:main:q5-1}
$\B(X)=(2^2, 3, 4)$, 
\item 
\label{thm:main:q5-2}
$A_X^3=1/12$ and $\g(X)\ge 5$,
\item 
\label{thm:main:q5-3}
$\p_2(X)\ge 2$.
\end{enumerate}
Then $X$ is isomorphic to a hypersurface as in Example~\xref{ex:X10}.
\item
\label{thm:main7a}
Let $X=X_{10}\subset\PP(1,2,3,4,5)$ be a hypersurface as in 
Example~\xref{ex:X10}.
Assume that every non-Gorenstein singularity of $X$ is a cyclic quotient 
and every Gorenstein singularity is either node or cusp. 
Furthermore, assume that the number of Gorenstein singularities is at most two.
Then $X$ is not rational.
\end{enumerate}
\end{theorem}
Let us briefly describe the structure of this paper.
Section \ref{sect:pre} is preliminary and in
Section \ref{sect:cb}, we present several basic facts about $\QQ$-conic bundles  which are used in the proofs of nonrationality
of $\QQ$-Fano threefolds. Then in Section~\ref{sect:sl} we describe the main tool used in our birational transformations: Sarkisov links.
Section~\ref{sect:tor} is devoted to
$\QQ$-Fano threefolds with nontrivial torsions in the Weil divisor class group. Main results there are Corollary~\ref{cor:q=3:tor} and Proposition~\ref
{prop:t}. In
Section~\ref{sect:q=5} we study a family of ``extremal'' $\QQ$-Fano threefolds of index~$5$ and prove~\ref{thm:main:5} of Theorem~\ref{thm:main}.
In Section~\ref{sect:67} we establish rationality of  $\QQ$-Fano threefolds of index~$6$ and ``most'' of $\QQ$-Fano threefolds of index~$7$.
The assertions~\ref{thm:main6} and~\ref{thm:main7} of Theorem~\ref{thm:main} will be proved there.
Finally, in Section~\ref{sect:q=7a} and~\ref{sect:q=7} we study nonrational $\QQ$-Fano threefolds of index~$7$.
The results there are not complete. This will be the subject of future research.

\subsection*{Acknowledgements.} The author o grateful to the referee for careful reading the manuscript and
 comments that helped to improve the presentation.

\section{Preliminaries}
\label{sect:pre}
\subsection{Notation}
\begin{itemize}
\item
$\sim$ (resp. $\qq$) denotes the linear (resp. $\QQ$-linear) equivalence of 
Weil divisors, 
we write $D_1\overset{P}{\sim} D_2$ if the linear equivalence $D_1\sim D_2$ 
holds in a neighborhood of the point $P$;
\item
$\Cl(X)$ denotes the Weil divisors class group of a normal variety and
$\Clt{X}$ denotes the torsion part of $\Cl(X)$;
\item 
$\Cl(X,P)$ denotes the local Weil divisors class group of an isolated singularity $P\in X$, that is, $\Cl(X,P)=\Cl(U)$ for a 
sufficiently small neighborhood 
of $P$;

\item
$\ind(X,P)$ is the Gorenstein index of $X$ of a $\QQ$-Gorenstein singularity 
$P\in X$ and
$\ind(X)$ be the global Gorenstein index of $X$, that is, 
$\ind(X):=\lcm\{\ind(X,P)\mid P\in X\}$;

\item 
$\B(X)$ is the collection of indices of singularities in the basket of a terminal threefold $X$ (see \ref{prop:Q-smoo});

\item
$\aw(X,P)$ is the axial weight of a threefold singularity (see 
Proposition~\ref{prop:Q-smoo});

\item
$\g(X):=\dim|-K_X|-1$ is the genus of a $\QQ$-Fano threefold $X$.
\end{itemize}

\subsection{Singularities}
Recall that a hypersurface $n$-dimensional singularity is called \textit{node} if
the Hessian matrix of its local equation is nondegenerate.
A hypersurface $n$-dimensional singularity is called
\textit{cusp} if its local equation
in some analytic coordinates has the form
$ \sum_{i=1}^{n} x_i^2+ x_{n+1}^3=0$. In both cases the blowup of this 
singularity is 
a resolution and the exceptional divisor is an irreducible quadric.

For the classification of threefold terminal singularities we refer to 
\cite{Mori:term-sing}
and \cite{Reid:YPG}. Basically, we need only two types of terminal 
singularities:
a threefold terminal singularity $X\ni P$ is said to be of type~\typeA{cA}{r} 
(resp.~\typeA{cAx}{4}) if it
is analytically isomorphic to the quotient
\begin{eqnarray}
\label{eq:cA/r}
\text{type \typeA{cA}{r}:}&& \{ x_1x_2+\phi(x_3^r,x_4)=0\} /\mumu_r(a,-a,1,0),\quad\gcd(r,a)=1,
\\
\label{eq:cAx/4}
\text{type \typeA{cAx}{4}:}&& \{ x_1^2+x_3^2+\phi(x_2,x_4)=0\} /\mumu_4(1,1,3,2), \quad \phi \in 
\mathfrak{m}^2,
\end{eqnarray}
where $\phi$ is a semi-invariant of the corresponding weight.

\begin{sproposition-definition}[\cite{Reid:YPG}]
\label{prop:Q-smoo}
Any threefold terminal singularity $X\ni P$ admits a small deformation 
$\mathfrak{X}\to \mathbf{D}$
over a disk $\mathbf{D}\ni 0$ such that the central fiber $\mathfrak{X}_0$ is 
isomorphic to $X$ and the nearby fibers $\mathfrak{X}_s$, $s\neq 0$ have
only terminal cyclic quotient singularities. Such a deformation is called 
\emph{$\QQ$-smoothing}.
The collection of cyclic quotient singularities of $\mathfrak{X}_s$ is called 
the \emph{basket} of $X\ni P$ and
the size of the basket is called the \emph{axial weight} of $X\ni P$.
It is denoted by $\aw(X,P)$. 
\end{sproposition-definition}

For our purposes only indices of the points in the basket are important. 
So, we will write $\B(X,P)=(r_1,\dots,r_n)$ if $r_1,\dots,r_n$ are indices of 
cyclic quotient singularities in the corresponding basket.
Thus $n=\aw(X,P)$. 
If $X$ is a threefold having only terminal singularities, then the basket of 
$X$ is 
the (disjoint) union of the baskets of singular points of $X$ and denote 
$\B(X)= \sqcup_{P\in X}\B(X,P)$.

\begin{sremark}
For a terminal singularity $X\ni P$ of index~$r>1$ we have 
\[
\B(X,P)=
\begin{cases}
(r,\dots,r) &\text{if $X\ni P$ is not of type \typeA{cAx}{4}},
\\
(4,2,\dots,2)&\text{if $X\ni P$ is of type \typeA{cAx}{4}}.
\end{cases}
\]
Moreover, 
for a singularity given by \eqref{eq:cA/r} or \eqref{eq:cAx/4} we have
\[
\aw(X,P):=\mult_0(\phi(0,t)).
\]
\end{sremark}

\begin{slemma}[{\cite[Lemma~5.1]{Kawamata:crep}}]
\label{lemma:K-index}
Let $(X\ni P)$ be a threefold terminal singularity and let 
$\Cl^{\mathrm{sc}}(X,P)$ be the subgroup of the \textup(analytic\textup) Weil 
divisor class group consisting of
Weil divisor classes which are $\QQ$-Cartier. Then the group 
$\Cl^{\mathrm{sc}}(X,P)$ is cyclic of order $\ind(X,P)$
and is generated by the canonical class $K_X$.
\end{slemma}

Recall that an \emph{extremal blowup} of a threefold $X$ with terminal
$\QQ$-factorial singularities
is a birational morphism $f:\tilde{X}\to X$
such that
$\tilde{X}$ also has only terminal $\QQ$-factorial singularities and
$\uprho(\tilde{X}/X)=1$.
In this situation the anticanonical divisor $-K_{\tilde X}$ must be $f$-ample.

\begin{slemma}
\label{lemma-discrepancies}
Let $X\ni P$ be a threefold terminal point of index~$r>1$
and
let $f: (\tilde X\supset E)\to (X\ni P)$ be an extremal blowup,
where $E$ is the exceptional divisor with $f(E)=P$. 
Write $K_{\tilde X}=f^*K_X+\frac kr E$.
\begin{enumerate}
\item 
If $X\ni P$ is a point of type other than $\mathrm{cA}/r$, $\mathrm{cD}/2$ or $\mathrm{cE}/2$ , then 
$k=1$.
\item 
If $X\ni P$ is of type $\mathrm{cA}/r$, then $n\equiv 0\mod k$, where 
$n=\aw(X,P)$.
\end{enumerate}
\end{slemma}

\begin{proof}
Follows from \cite[Theorems~1.2 and~1.3]{Kawakita:hi} (see 
\cite[Lemma~2.6]{P:2013-fano}).
\end{proof}

\subsection{$\QQ$-Fano threefolds}
For a Fano variety $X$ with at worst log terminal singularities we define its 
\emph{Fano} and
\emph{$\QQ$-Fano indices} as follows:

\begin{equation}
\label{eq:deq-q}
\begin{array}{lll}
\qW(X)&:=&\max \{ q\in\ZZ \mid -K_X\sim q A,\ \text{$A$ is a Weil divisor}\},
\\[5pt]
\qQ(X)&:=&\max \{ q\in \ZZ \mid -K_X\qq q A,\ \text{$A$ is a Weil divisor}\}.
\end{array}
\end{equation}
Clearly, $\qW(X)$ divides $\qQ(X)$, and $\qW(X)=\qQ(X)$ if the group $\Cl(X)$ 
is torsion free.
The \textit{fundamental divisor} of $X$ is a Weil divisor $A_X$ such that 
\begin{equation*}
% \label{eq:def-AX}
-K_X\qq \qQ(X) A_X.
\end{equation*} 
Note that if $\Clt{X}\neq 0$, then the class of $A_X$
is not uniquely defined modulo linear equivalence. 
However, in the case $\qQ(X)=\qW(X)$ we always take $A_X$
so that 
\begin{equation*}
\label{eq:def-AXa}
-K_X\sim \qW(X)A_X. 
\end{equation*}
The \textit{Hilbert series} of a $\mathbb{Q}$-Fano threefold $X$ is the 
following formal power series \cite{Altinok2002}:
\[
\h_X(t)= \sum_{m\ge 0} \h^0 (X, mA_X )\cdot t^m.
\]
It is computed by using the orbifold Riemann-Roch formula \cite{Reid:YPG}. If 
the group $\Cl(X)$ contains an element $T$ of $N$-torsion, we define 
\emph{$T$-Hilbert series} $\h_X(t,\sigma)\in 
\ZZ[[t,\sigma]]/(\sigma^N-1)$
as follows:
\[
\h_X(t,\sigma)= \sum_{m\ge 0} \sum_{j=0}^{N-1}\h^0 (X, mA_X+jT )\cdot 
t^m\sigma^j.
\]
Obviously, the above definition depends on the choice the class of $A_X$ in 
$\Cl(X)$. Typically calculating $\h_X(t)$ or $\h_X(t,\sigma)$ for our purposes 
we need only 
a few initial terms of the series.

\begin{sproposition}
\label{prop:6-7}
Let $X$ be a $\QQ$-Fano threefold such that $\qQ(X)=6$ or $7$ and $\dim 
|3A_X|\le 0$. 
Then the numerical invariants of $X$ are described by Table~\xref{tab:1}.
\begin{table}[H]
\setcounter{NN}{0}
\renewcommand{\arraystretch}{1.2}
\rm
\begin{tabularx}{0.9\textwidth}{|l|l|l|l|r|r|r|r|r|r|l| l| X|}
\hline
&$A_X^3$ & $\B(X)$ & $\g(X)$ &\multicolumn{6}{c|}{$\dim|k A_X|$}&\cite{GRD} & 
$\exists?$ &Ref
\\\hhline{~~~~------}
&&&&1 & 2 & 3 & 4 & 5 & 6 &&&
\\\hline
\multicolumn{13}{|c|}{$\qQ(X)=6$}
\\\hline
\nr\label{Case:q=6:B=5-17}& $2/85$ & $(5,17)$ & $1$ & $0$ & $0$ & $0$ & $0$ & 
$1$ & & 41465 & $?$ &
\\
\nr\label{Case:q=6:B=7-11}& $2/77$ & $(7,11)$ & $2$ & $-1$ & $0$ & $0$ & $1$ & 
$1$ & & 41461 & $?$&
\\
\nr\label{Case:q=6:B=5-11}& $1/55$ & $(5,11)$ & $1$ & $0$ & $0$ & $0$ & $0$ & 
$1$ & & 41464 & $+$&
\cite{Coughlan-Duca2018}, \cite{Coughlan-Duca:18:table4}
\\\hline
\multicolumn{13}{|c|}{$\qQ(X)=7$}
\\\hline
\nr
\label{Case:q=7:B=2-2-2-3-4-5}& $1/60$ & $(2^3, 3, 4, 5)$ & $2$ & $-1$ & $ 0$ & 
$ 0$ & $1$ & $ 1$ & $ 2$&41474 & $+$ &\cite{Brown-Suzuki-2007j}
\\
\nr
\label{Case:q=7:B=2-2-2-5-8}& $1/40$ & $(2^3, 5, 8)$ & $3$ & $ -1$ & $0$ & $ 0$ 
& $ 1$ & $2$ & $ 3$&41475& $+$
&\cite{Coughlan-Duca2018}, \cite{Coughlan-Duca:18:table4}
\\
\nr
\label{Case:q=7:B=3-8-9}& $1/72$ & $(3, 8, 9)$ & $1$ & $-1$ & $-1$ & $0$ & $0$ 
& $1$ & $1$&41472&$?$&
\\
\nr
\label{Case:q=7:B=2-6-10}& $1/30$ & $(2, 6, 10)$ & $5$ & $0$ & $0$ & $ 0$ & $1$ 
& $ 2$ & $4$&41479 &$+$&
Example~\ref{ex:tor}
\\
\nr
\label{Case:q=7:B=2-313}& $1/78$ & $(2, 3, 13)$ & $1$ & $0$ & $0$ & $0$ & $0$ & 
$0$ & $1$ &41478& $?$&
\\\hline
\end{tabularx}
\caption{$\QQ$-Fano threefolds with $\qQ(X)=6$ or $7$ and $\dim|3A_X|\le 0$}
\label{tab:1}
\end{table}
\noindent
In all cases, except for~\xref{Case:q=7:B=2-6-10}, the group $\Cl(X)$ is 
torsion free.
In the case~\xref{Case:q=7:B=2-6-10} the group $\Cl(X)$ may have $2$-torsion 
and then
\begin{equation}
\label{eq:h::q=7:B=2-6-10}
\h_X(t)=1+t+t^2 (1+\sigma)+t^3(1+2\sigma)+t^4(2+3\sigma)+t^5(3+4\sigma)+\cdots.
\end{equation} 
\end{sproposition}

\begin{sproposition}[see {\cite[\S~3]{P:2019:rat:Q-Fano}}]
\label{prop:tor}
Let $X$ be a $\QQ$-Fano threefold with $\qQ(X)\ge 5$ and \mbox{$\Clt{X}\neq 0$}.
Then $\qW(X)=\qQ(X)=5$ or $7$ and 
the numerical invariants of $X$ are described by Table~\xref{tab:2}.
\begin{table}[H]
\rm
\setcounter{NN}{0}
\renewcommand{\arraystretch}{1.2}
\begin{tabularx}{0.99\textwidth}{|l|l|l|c|X|}
\hline
&\multicolumn{1}{c|}{$A_X^3$} & \multicolumn{1}{c|}{$\B(X)$} & 
\multicolumn{1}{c|}{$\g(X)$}
&\multicolumn{1}{c|}{$\h_X(t,\sigma)$ }
\\\hline
\multicolumn{5}{|c|}{\textbf{$\qQ(X)=7$}, $\Clt{X}\simeq \ZZ/2\ZZ$}
\\\hline
\nr
\label {tab:h:q=7t:d=1/24}
& $1/24$&
$(2^2, 3, 4, 8)$&
$6$&
$1+ t \sigma +t^2 +t^2 \sigma + 2 t^3 + 2 t^3 \sigma + 3 t^4 + 3 t^4 \sigma + 4 
t^5 + 4 t^5 \sigma + \cdots$
\\
\nr
\label {tab:h:q=7t:d=1/30} &$1/30$&
$(2, 6, 10)$&
$5$&
$1+t +t^2 +t^2 \sigma +t^3 + 2 t^3 \sigma + 2 t^4 + 3 t^4 \sigma + 3 t^5 + 4 
t^5 \sigma + \cdots $
\\\hline
\multicolumn{5}{|c|}{\textbf{$\qQ(X)=5$, $\Clt{X}\simeq \ZZ/3\ZZ$}}
\\\hline
\nr 
\label{tab:h:q=5t3:d=1/18} 
& 
$1/18$& $(2, 9^2)$ & $2$ &
$1+t + t^2 +t^2 \sigma +t^2 \sigma^2 +t^3 + 2 t^3 \sigma + 2 t^3 \sigma^2 + 
\cdots$
\\\hline
\multicolumn{5}{|c|}{\textbf{$\qQ(X)=5$, $\Clt{X}\simeq \ZZ/2\ZZ$}}
\\\hline
\nr 
\label{tab:h:q=5t2:d=1/6}
& 
$1/6$& $(2, 4^2, 6)$ & $10$ & 
$1+t +t \sigma + 2 t^2 + 3 t^2 \sigma + 4 t^3 + 5 t^3 \sigma + 8 t^4 + 7 t^4 
\sigma + 12 t^5 + 11 t^5 \sigma + \cdots$
\\\nr
\label{tab:h:q=5t2:d=1/8}
& 
$1/8$& $(2^2, 4, 8)$ & $7$ &
$1+t +t \sigma + 2 t^2 + 2 t^2 \sigma + 3 t^3 + 4 t^3 \sigma + 6 t^4 + 6 t^4 
\sigma + 9 t^5 + 9 t^5 \sigma + \cdots$
\\
\nr 
\label{tor:q=5:1/12}
& $1/12$& $(4^2, 12)$ & $4$ &
$1+t + t^2 +t^2 \sigma + 2 t^3 + 2 t^3 \sigma + 4 t^4 + 4 t^4 \sigma + 6 t^5 + 
6 t^5 \sigma + \cdots$
\\
\nr 
\label{tor:q=5:1/28}
& $1/28$& $(2, 4, 14)$ & 
$1$ &
$1+t + t^2 + t^3 +t^3 \sigma + 2 t^4 + 2 t^4 \sigma + 3 t^5 + 3 t^5 \sigma 
+\cdots$ 
\\\hline
\end{tabularx}
\caption{$\QQ$-Fano threefolds with $\qQ(X)\ge 5$ and $\Clt{X}\neq 0$}
\label{tab:2}
\end{table}
\end{sproposition}

\section{$\QQ$-conic bundles}
\label{sect:cb}
In this section we collect basic facts on threefold Mori fiber spaces with
one-dimensional
fibers. For more detailed information and references
we refer to~\cite{MP:cb1}, \cite{MP:cb2}, \cite{P:ICM}.

\begin{definition}[\cite{MP:cb1}]
A \emph{$\QQ$-conic bundle} is a contraction $\pi:Y\to S$
from a threefold to a surface such that $Y$ is normal and has only terminal
singularities, all fibers are one-dimensional, and $-K_Y$ is $\pi$-ample.
We say that a $\QQ$-conic bundle $\pi:Y\to S$ is \emph{extremal} if $Y$ is
$\QQ$-factorial and
the relative Picard number $\uprho(Y/S)$ equals $1$.
We say that $\pi:Y\to S$ is a \emph{$\QQ$-conic bundle germ} over a point
$o\in S$
if $S$ (resp. $Y$) is regarded as a germ at $o$ (resp. along $\pi^{-1}(o)$).
The \emph{discriminant divisor} of a $\QQ$-conic bundle $\pi:Y\to S$ is the
curve $\Delta_\pi\subset S$ that is the union of one-dimensional
components of the set
\[
\{s\in S\mid\text{ $\pi$ is not smooth over $s$}\}.
\]
We say that a fiber $\pi^{-1}(o)$ is \emph{standard}, if $Y$ is smooth along
$\pi^{-1}(o)$, in other words, if $\pi$ is a standard conic bundle in a 
neighborhood of $\pi^{-1}(o)$.
\end{definition}

\begin{theorem}[{\cite[Theorem~1.2.7]{MP:cb1}}]
\label{thm:base-surface}
Let $\pi: Y\to S$ be a $\QQ$-conic bundle. Then the singularities of the base
$S$ are at worst
Du Val of type~\type{A}.
\end{theorem}

\begin{lemma}
\label{lemma:DP}
Let $S$ be a Du Val del Pezzo surface. Then $\dim|k A_S|\ge k-1$ for $k>0$. 
Moreover, if $k<K_S^2<8$, then the equality $\dim|k A_S|=k-1$ holds.
\end{lemma}
\begin{proof}
By \cite[III, Theorem~1]{Demazure1980} a general member $C\in |-K_S|$ is a 
smooth elliptic curve.
By the Kawamata-Viehweg vanishing theorem $H^1(S,\OOO_S(kA_S+K_S))=0$.
Hence we have the following exact sequence
\[
0\longrightarrow H^0(S,\OOO_S(kA_S-C)) 
\longrightarrow H^0(S,\OOO_S(kA_S)) 
\longrightarrow H^0(C,\OOO_C(kA_S)) \longrightarrow 0.
\]
Now the assertion follows from the Riemann-Roch formula on $C$.
\end{proof}

\begin{scorollary}
\label{cor:base-surface}
Let $\pi: Y\to S$ be a $\QQ$-conic bundle, where the variety $Y$ is
rationally connected and $\rk\Cl(Y)=2$. Then $S$ is a del Pezzo surface with
Du Val singularities of type~\type{A}.
Furthermore, assume that $\Cl(Y)\simeq\ZZ\oplus\ZZ$. Then $\Cl(S)\simeq\ZZ$ and 
$S$ is
isomorphic to one of the following four surfaces described in Table~\xref{tab3},
where 
$S_{\mathrm{DP}_5}$ is a quasi-smooth hypersurface of degree $6$ in 
$\PP(1,2,3,5)$.
\begin{table}[H]
\setcounter{NN}{0}
\renewcommand{\arraystretch}{1.2}
\rm
\begin{tabularx}{0.9\textwidth}{|l|m{0.1\textwidth}|m{0.1\textwidth}|
m{0.1\textwidth}|r|X|X|X|X|X|X| }
\hline
$S$&$K_S^2$ & $\qW(S)$ &$A_S^2$& $\Sing(S)$ &\multicolumn{6}{c|}{$\dim|k A_S|$}
\\\hhline{~~~~~-----~}
&&&&&1 & 2 & 3 & 4 & 5 
\\\hline
$\PP^2$ & $9$ & $3$ & $1$ &$\varnothing$& $2$ & $5$ & $9$ & $14$ & $20$
\\
$\PP(1^2,2)$& $8$& $4$& $1/2$ &\types{A}{1} &$1$ & $3$ & $5$ & $8$ & $11$
\\
$\PP(1,2,3)$& $6$& $6$& $1/6$&\types{A}{1}\types{A}{2} &$0$ & $1$ & $2$ & $3$ & 
$4$
\\
$S_{\mathrm{DP}_5}$& $5$& $5$& $1/5$&\types{A}{4} &$0$ & $1$ & $2$ & $3$ & $5$
\\\hline
\end{tabularx}
\caption{Del Pezzo surfaces with
singularities of type~\type{A} and $\Cl(S)\simeq\ZZ$}\label{tab3}
\end{table}
\end{scorollary}

\begin{lemma}[see e.g. {\cite[Lemma~4.4]{P:P11223}}]
\label{lemma:cb}
Let $\pi: Y\to S$ be a $\QQ$-conic bundle, where the variety $Y$ is projective,
and let $\Delta_{\pi}\subset S$ be its discriminant curve.
Let
$H$ be an ample divisor on $S$ and let $F:=\pi^{-1}( H)$. Then
\begin{eqnarray}
\label{eq:cb:a}
K_Y\cdot F^2&=& -2 H^2,
\\
\label{eq:cb:b}
K_Y^2\cdot F&=&-4K_S\cdot H-H\cdot\Delta_{\pi}.
\end{eqnarray}
\end{lemma}

Below we provide several ``basic'' examples of $\QQ$-conic bundles.
We are interested in the local structure of them near the singular fiber.
Denote by $\zeta_r$ a primitive $r$-th root of unity.

\begin{example}[Type~\typec{IF_1}]
\label{ex:ODP}
The variety $Y$ is given in $\PP^2_{y_1,y_2, y_3}\times
\CC^2_{u,v}$ by the equation
\[
y_1^2+y_2^2+uvy_3^2=0,
\]
and $\pi\colon Y\to S$ is the
projection to
$S:=\CC^2$.
The singular locus of $Y$ consists of one ordinary double point,
$\Delta_\pi=\{uv=0\}$, and the singular fiber $C:=\pi^{-1}(0)_{\mathrm{red}}$
is a pair of lines.
\end{example}

\begin{example}[Type~\typeci{T}{r}]
\label{ex:T}
The variety $Y$ is the quotient of $\PP^1_{y_1, y_2}\times
\CC^2_{u,v}$ by the $\mumu_r$-action
\[
(y_1, y_2;\, u,v)\longmapsto(y_1,\zeta_r\, y_2;\,\zeta_r^a\, u,\zeta_r^{-a}\,
v),
\]
where $\gcd (r,a)=1$, and $\pi\colon Y\to S$ is the
projection to
$S:=\CC^2/\mumu_r$. The singular locus of $Y$ consists of two (terminal) cyclic
quotient singularities of types $\frac1r(1,a,-a)$ and
$\frac1r(-1,a,-a)$. The singularity of $S$ at the origin is of type
\types{A}{r-1}.
In this case $\Delta_\pi=\varnothing$ and for the singular fiber
$C:=\pi^{-1}(0)_{\mathrm{red}}$ we have
$-K_Y\cdot C=2/r$.
\end{example}

\begin{example}[Type~\typeci{k2A}{r}]
\label{ex:k2A}
The variety $Y$ is the quotient of the hypersurface
\[
Y'=\{ y_1^2+uy_2^2+vy_3^2=0\}\subset\PP^2_{y_1,y_2,y_3}\times\CC^2_{u,v}
\]
by the $\mumu_{r}$-action
\[
(y_1,y_2,y_3;\, u,v)\longmapsto (\zeta_r^{a}\, y_1,\zeta_r^{-1}\, y_2,y_3;\,
\zeta_r\, u,\zeta_r^{-1}\, v),
\]
where $r=2a+1$ and $\pi\colon Y\to S$ is the
projection to
$S:=\CC^2/\mumu_r$.
The singular locus of $Y$ consists of two (terminal) cyclic quotient
singularities of types $\frac 1r(a,-1,1)$ and $\frac 1r(a+1,1,-1)$.
The singularity of $S$ at the origin is of type~\types{A}{r-1}.
In this case $\Delta_\pi=\{uv=0\}/\mumu_r$ and for the singular fiber
$C:=\pi^{-1}(0)_{\mathrm{red}}$ we have
$-K_Y\cdot C=1/r$.
\end{example}

\begin{example}[Type~\typec{ID_1^\vee}]
\label{ex:ID}
The variety $Y$ is the quotient of the hypersurface
\[
\{y_1^2+y_2^2+uv y_3^2=0\}\subset\PP^2_{y_1,y_2,y_3}\times
\CC^2_{u,v}\},
\]
by the $\mumu_2$-action
\[
(y_1,y_2,y_3;\, u,v)\longmapsto (-y_1,y_2,y_3;\, -u,-v).
\]
Then $Y$ has a unique singular point that is moderate of index~$2$ and axial 
weight $2$.
The singularity of $S$ at the origin is of type~\types{A}1.
In this case $\Delta_\pi=\{uv=0\}/\mumu_m$ and for the singular fiber
$C:=\pi^{-1}(0)_{\mathrm{red}}$ we have
$-K_Y\cdot C=1$.
\end{example}

\begin{remark}
In all cases  \typec{IF_1}, \typeci{T}{r}, \typeci{k2A}{r} \typec{ID_1^\vee} the pair $(S,\Delta_{\pi})$ is lc.
\end{remark}

The following fact is a direct consequence of \cite[Theorem~1.2]{MP:cb1} and
\cite[Theorem~1.3]{MP:cb2}.
\begin{proposition}
\label{prop:cb-index2}
Let $\pi:Y\to S\ni o$ be a $\QQ$-conic bundle germ, where
the singularities of $Y$ are moderate of indices $\le 3$.
Assume that $S\ni o$ is singular. Then $S\ni o$ is of type~\type{A_1} 
or~\type{A_2}. Furthermore,
\begin{enumerate}
\item
if
$S\ni o$ is a
singularity of type~\type{A_1}, then $\pi$ is of type~\typeci{T}{2} 
or~\typec{ID_1^\vee};
\item
if
$S\ni o$ is a
singularity of type~\type{A_2}, then $\pi$ is of type~\typeci{T}{3} or~\typeci{k2A}{3}.
\end{enumerate}
In particular, the pair $(S,\Delta_\pi)$ is lc.
\end{proposition}

\begin{theorem}[cf. {\cite[Sect. 11]{P:rat-cb:e}}]
\label{thm:SL-CB}
Let $\pi: Y\to S\ni o$ be a $\QQ$-conic bundle germ of one of the types
\typec{IF_1}, \typeci{T}{r}, \typeci{k2A}{r}, or~\typec{ID_1^\vee}.
Let $P\in Y$ be a singular point, let $r$ be its index, and let $p:\check Y\to 
Y$ be
an extremal blowup
of $P$ with discrepancy $1/r$.
Then $p$ can be completed to a type~\typem{I} Sarkisov link
\begin{equation}
\label{eq:SL-CB}
\vcenter{
\xymatrix@R=1.2em@C=3em{
\check{Y}\ar[d]_{p}\ar@{-->}[r]^{\chi}& Y'\ar[d]^{\pi'}
\\
Y\ar[d]_{\pi} & S'\ar[ld]_{\theta}
\\
S
} }
\end{equation}
where $\chi$ is a birational transformation that is isomorphism in codimension
one,
$\pi'$ is a $\QQ$-conic bundle, and $\theta$ is a crepant contraction with
$\uprho(S'/S)=1$ in the cases~\typec{T}, \typec{k2A},
\typec{ID_1^\vee},
and $\theta$ is the blowup of $o$ in the case~\typec{IF_1}.
\begin{enumerate}

\item
\label{thm:SL-CB:a}
If $\pi: Y\to S\ni o$ is of type~\typeci{T}{r}, then $\pi': Y'\to S'$ has two
(or
one) singular fibers
and corresponding $\QQ$-conic bundle germs are of type~\typeci{T}{a} and
\typeci{T}{r-a}.
\item
\label{thm:SL-CB:b}
If $\pi: Y\to S\ni o$ is of type~\typeci{k2A}{r}, then $\pi': Y'\to S'$ has two
(or
one) singular fibers
and corresponding $\QQ$-conic bundle germs are of type~\typeci{k2A}{r-2} and
of type~\typec{ID_1^\vee}.
In this case $\Delta_{\pi'}=\theta^*\Delta_{\pi}$.
\item
\label{thm:SL-CB:c}
If $\pi: Y\to S\ni o$ is of type~\typec{ID_1^\vee}, then $\pi': Y'\to S'$
is a standard conic bundle and $\Delta_{\pi'}=\theta^*\Delta_{\pi}$.
\item 
\label{thm:SL-CB:o}
If $\pi: Y\to S\ni o$ is a Gorenstein $\QQ$-conic bundle germ such that the 
singularities
of $Y$ are either nodes or cusps, then so are
the singularities of 
$Y'$,  $\pi': 
Y'\to S'$ is a Gorenstein conic bundle with $|\Sing(Y')| = |\Sing(Y)|-1$,
and $\Delta_{\pi'}$ is the proper transform of~$\Delta_{\pi}$.
\end{enumerate}

In all cases we have
\begin{equation}
\label{eq:log-crepant}
\textstyle
K_{S'}+\frac12\Delta_{\pi'} =\theta^*\left(K_S+\frac12\Delta_{\pi}\right).
\end{equation}
\end{theorem}

\begin{proof}
For \ref{thm:SL-CB:a} and \ref{thm:SL-CB:b} we refer to \cite[\S~11]{P:rat-cb:e}
and for \ref{thm:SL-CB:c} we refer to \cite[Theorem~4.14]{P:P11223}.
The assertion \ref{thm:SL-CB:o} is proved in the same style (see 
\cite[Construction~11.1]{P:rat-cb:e}
and \cite[Lemma~8]{Avilov:cb}). In this situation 
$\Delta_{\pi'}$ is the proper transform of~$\Delta_{\pi}$ by 
\cite[Lemma~10.10]{P:rat-cb:e}.
\end{proof}

\begin{theorem}
\label{thm:cb}
Let $\varphi: Y\to S$ be a $\QQ$-conic bundle such that
the singularities of $Y$ are moderate of indices $\le 3$
and $S$ is a weak del Pezzo surface with $d:=(-K_S)^2$.
Assume that 
\begin{equation}
\label{eq:K-Delta1}
\Delta_{\varphi}\sim -2K_{S}.
\end{equation}
Furthermore, assume that outside the fibers over $\Sing(S)$
they are nodes or cusps, and the number of the singularities on $Y\setminus 
\varphi^{-1}(\Sing(S))$ is at most $d-4$.
Then $Y$ is not rational.
\end{theorem}

\begin{proof}\renewcommand{\qedsymbol}{}
Let $P_1,\dots,P_n\in Y$ be Gorenstein singular points and let 
$o_i:=\varphi(P_i)$ for $i=1,\dots,n$. 
First of all note that by Proposition~\ref{prop:cb-index2} the $\QQ$-conic
bundle $\varphi$ is of type~\typeci{T}{2} or~\typec{ID_1^\vee} 
over each \type{A_1}-point of $S$ and it is of type~\typeci{T}{3} 
or~\typeci{k2A}{3}
over each \type{A_2}-point.
Hence the points $o_1,\dots,o_n\in S$ are smooth.
By Theorem~\ref{thm:SL-CB} there exists the following composition of Sarkisov 
links:
\begin{equation}
\label{eq:SL-q=5:CB-ns}
\vcenter{
\xymatrix@C=3em{
Y\ar[d]_{\varphi} &\tilde Y\ar[d]_{\pi}\ar@{-->}[l]_{\tau}
\\
S &\tilde S\ar[l]_{\mu}
} }
\end{equation}
where $\tau$ is a birational map, $\mu$ is the composition of the minimal
resolution and blowups of the (possibly infinite near) points $o_1,\dots,o_n$, 
and $\pi$ is a standard conic bundle.
By~\eqref{eq:log-crepant} and~\eqref{eq:K-Delta1} we have
\begin{equation}
\label{eq:K-Delta2}
\Delta_{\pi}\sim -2K_{\tilde S}.
\end{equation}
Moreover, $\Delta_\pi$ is a nodal curve, hence the pair $(\tilde S, \Delta_\pi)$ is lc.
\end{proof}

We need two lemmas.
\begin{slemma}
\label{lemma:surf0}
Let $S$ be a projective  smooth rational surface such that $|-2K_S|\neq \varnothing$.
Suppose that the linear system $|-2K_S|$ contains a reduced divisor $D$
such that for any irreducible component $C\subset D$ that is a smooth 
rational curve
we have $(D-C)\cdot C>2$. Then $-K_S$ is nef.
\end{slemma}

\begin{proof}
Assume that for some irreducible curve $ C\subset S$ we have
$K_S \cdot C>0$. Then $ D \cdot C<0$, $ C$ 
is a component of $ D$, and $ C^2<0$. Put $D':= D-C$. Then
$D'$ is effective and does not contain $ C$ as a component.
We have
\[\textstyle
2\pa( C)-2= K_S \cdot C+ C^2 =\frac 12 
C^2-\frac 12 D' \cdot C<0.
\]
Hence, $\pa( C)=0$, $ C\simeq \PP^1$, and so 
\[
4= D' \cdot C- C^2 > D' \cdot C- C^2 + 
D \cdot C =2 D' \cdot C.
\]
This contradicts our assumption.
\end{proof}

\begin{slemma}
\label{lemma:surf}
Let $S$ be a smooth surface such that $-K_S$ is nef and $K_S^2\ge 4$,
and let $D\in |-2K_S|$ be a reduced divisor on $S$.
Suppose that there exists a decomposition $D=D'+D''$, where $D'$, $D''$ are
effective divisors and $D'\cdot D''=2$.
Then
either $-K_S\cdot D'=0$ or $-K_S\cdot D''=0$.
\end{slemma}

\begin{proof}
Assume that $-K_S\cdot D'>0$ and $-K_S\cdot D''>0$.
We have
\[
(-2 K_S)^2=D'^2+D''^2+4\ge 16,
\]
hence we may assume that $D''^2>0$.
Then by the Hodge index theorem $4=(D'\cdot D'')^2\ge D'^2\, D''^2$.
This implies that $D'^2\le 0$. Then
\[
2=D'\cdot D''=D'\cdot (-2K_S -D')=-2K_S\cdot D' -D'^2\ge 2.
\]
So we obtain $K_S\cdot D'=-1$ and $D'^2=0$.
On the other hand,
by the genus formula the number $K_S\cdot D'+D'^2$ must be even, a 
contradiction.
\end{proof}

\begin{proof}[Proof of Theorem~\xref{thm:cb} \textup(continued\textup)]
Let $d:=K_S^2$.
We use the diagram~\eqref{eq:SL-q=5:CB-ns}.
Then $\tilde S$ is a smooth weak del Pezzo surface of degree $d-n\ge 4$ by 
Lemma~\ref{lemma:surf0}, hence
the linear system $|-K_{\tilde S}|$ is base point free and defines a crepant 
contraction
$\psi :\tilde S\to\bar S$ \cite[Corollary~4.5(i)]{Hidaka-Watanabe}.
Let $\bar\Delta:=\psi_*\Delta_{\tilde{\pi}}$.
By~\eqref{eq:K-Delta2} we have
\begin{equation*}
\bar\Delta\sim -2K_{\bar S}.
\end{equation*}
For any $\psi$-exceptional curve $\tilde E$ we have $K_{\tilde S}\cdot\tilde 
E=\Delta_{\tilde{\pi}}\cdot\tilde E=0$.
Hence, $\psi$ is log crepant with respect to $K_{\tilde 
S}+\Delta_{\tilde{\pi}}$,
the pair $(\bar S,\bar\Delta)$ is lc, and the curve $\bar\Delta$ has at worst 
nodal
singularities.
We claim that the curve $\bar\Delta$ satisfies the condition (S) of
\cite[Main Theorem]{Shokurov:Prym}.
Indeed, assume that $\bar\Delta=\bar\Delta'\cup \bar\Delta''$ so that
$\#(\bar\Delta'\cap \bar\Delta'')=2$. Let $\tilde\Delta'=\psi^{-1}(\bar\Delta')$ and $\tilde\Delta''=\psi^{-1}(\bar\Delta)-\tilde\Delta'$.
Then $\tilde\Delta'\cdot \tilde\Delta''=2$ and by Lemma~\ref{lemma:surf} either $\tilde\Delta'$ or $\tilde\Delta''$ is contracted by $\psi$,
a contradiction.

Let $\tilde\tau:\hat\Delta_{\tilde{\pi}}\to
\Delta_{\tilde{\pi}}$ be the double cover associated to $\tilde{\pi}$.
Then $\tilde\tau$ induces a double cover $\bar\tau:\hat\Delta\to\bar\Delta$
that coincides with $\tilde\tau$ over
$\Delta_{\pi}\setminus\Exc(\psi)$ and by \cite[Corollary~3.16]{Shokurov:Prym} 
we have a natural isomorphism of Prym varieties
\[
\Pr(\hat\Delta_{\tilde{\pi}}/\Delta_{\tilde{\pi}} )\simeq\Pr(\hat
\Delta/\bar\Delta).
\]
By the adjunction the divisor $K_{\bar\Delta}=-K_{\bar S}|_{\bar\Delta}$ is 
very ample.
Hence
the curve $\bar\Delta$ is not hyperelliptic.
Moreover, since $\bar\Delta\in |-2K_{\bar S}|$ and $\bar S$ is a du Val del 
Pezzo surface
of degree $d-n\ge 4$, the canonical model of $\bar\Delta$ is an intersection of 
quadrics by \cite[Theorem~4.4]{Hidaka-Watanabe}.
Hence,
$\bar\Delta$ is not trigonal nor quasi-trigonal, and also
$\bar\Delta$ is not a plane quintic (cf. \cite[Lemma~6.5]{P:P11223}).
Then by \cite[Main Theorem]{Shokurov:Prym} the Prym variety $\Pr(\hat
\Delta/\bar\Delta)$ is not a sum of Jacobians of curves, hence $\tilde Y$ is 
not rational (see e.g. \cite{Beauville:Prym}).
\end{proof}

\section{Sarkisov link}
\label{sect:sl}
Below we will frequently use the following Sarkisov link. For more detailed explanation of this construction 
we refer to~\cite{Alexeev:ge},  
\cite{P:2010:QFano}, \cite{P:2013-fano}, \cite{P:2016:QFano7}.

\subsection{Notation}
\label{subsect:SL}
Let $X$ be a non-Gorenstein $\QQ$-Fano threefold of $\QQ$-Fano index~$q=\qQ(X)>1$.
Let $\MMM$ be a nonempty linear system without fixed components such that
$\MMM\qq nA_X$ with $n<q$.
Let $c:=\operatorname {ct} (X,\MMM)$ be the canonical threshold of the pair 
$(X,\MMM)$.
We assume that $c\le 1$ (see Lemma~\ref{lemma:ct} below).
According to {\cite[Proposition~2.10]{Corti95:Sark}} (see also 
{\cite[Claim~4.5.1]{P:G-MMP}})
there exists an extremal blowup $f:\tilde{X}\to X$
that is crepant with respect to $K_X+c\MMM$.
Thus $\uprho (\tilde{X})=2$ and $-(K_{\tilde X}+c\tilde\MMM)$ 
is nef and big, where $\tilde \MMM$ is the proper transform of $\MMM$.
Run the log minimal model program on $\tilde{X}$ with 
respect to $K_{\tilde{X}}+c\tilde\MMM$
(see e.g. \cite[4.2]{Alexeev:ge} or \cite[12.2.1]{P:G-MMP}).
We obtain the following Sarkisov link:
\begin{equation}
\label{diagram-main}
\vcenter{
\xymatrix@C=19pt{
&\tilde{X}\ar@{-->}[rr]^{\chi}\ar@/_0.4em/[dl]_{f} 
&&\bar{X}\ar@/^0.4em/[dr]^{\bar f}
\\
X &&&&\hat{X}
} 
}
\end{equation}
where $\chi$ is an isomorphism in codimension $1$,
the variety $\bar{X}$ also has only terminal $\QQ$-factorial singularities,
$\uprho (\bar{X})=2$, and
$\bar{f}:\bar{X}\to\hat{X}$ is an extremal $K_{\bar{X}}$-negative Mori 
contraction
that can be either divisorial or fiber type. Denote by $\bar \MMM$ the proper transform of 
$\tilde \MMM$ on $\bar X$.

Let $E\subset \tilde X$ be the $f$-exceptional divisor and let $\bar E\subset 
\bar X$ be 
its proper transform.
Write
\begin{equation}
\label{equation-1}
K_{\tilde{X}}\qq f^*K_{X}+\alpha E,\quad
\tilde\MMM\qq f^* \MMM-\beta E.
\end{equation} 
Then 
\[
c= \frac {\alpha}{\beta}.
\]

\begin{slemma}[{\cite[Lemma~4.2]{P:2010:QFano}}]
\label{lemma:ct}
Let $P\in X$ be a point of index~$r>1$ such that $\MMM$ is not Cartier at $P$. Write
$\MMM\overset{P}{\sim} -mK_X$, where $0<m<r$.\footnote{Recall that $D_1\overset{P}{\sim} D_2$ means 
that 
the linear equivalence $D_1\sim D_2$ holds in a neighborhood of $P$.
}
Then
$\ct(X,\MMM)\le 1/m$. 
Therefore,
\begin{equation}
\label{eq:ct:beta-alpha}
\beta\ge m\alpha\qquad\text{and}\qquad q\beta-n\alpha\ge\alpha>0.
\end{equation} 
\end{slemma}

For a positive integer $k$ put $\MMM_k:=|kA_{X}|$. 
Thus $\MMM_n=\MMM$ if $\MMM$ is a complete linear system.
Let $\tilde \MMM_k$ and 
$\bar \MMM_k$
be proper transforms of $\MMM_k$ on $\tilde X$ and $\bar X$, respectively.
If $\MMM_k \neq \varnothing$, write
\[
\tilde \MMM_k\qq f^*\MMM_k-\beta_k E. 
\]

\subsection{Birational case}
\label{case-bir}
Assume that the contraction $\bar{f}$ is birational. Then $\hat{X}$ is a $\QQ 
$-Fano threefold.
In this case, denote by $\bar{F}$ the
$\bar{f}$-exceptional divisor, by
$\tilde F\subset\tilde{X}$ its proper transform, and $F:=f(\tilde F)$.
The divisor $\bar E$ is not contracted by $\bar f$, i.e.~$\bar{E}\neq\bar{F}$ 
(see e.g. \cite[Claim~4.6]{P:2010:QFano}).
Let $\hat \MMM:=\bar f_*\bar \MMM$, $\hat \MMM_k:=\bar f_*\bar \MMM_k$, $\hat E:=\bar f_*\bar E$, and let 
$A_{\hat{X}}$ be a fundamental divisor on $\hat{X}$.
Write
\[
F\qq d A_X,\qquad\hat{E}\qq eA_{\hat{X}},\qquad\hat\MMM\qq sA_{\hat{X}},\qquad\hat\MMM_k\qq s_kA_{\hat{X}},
\]
where $d,\, e\in\ZZ_{>0}$,\ $s,\, s_k\in\ZZ_{\ge0}$. 
For short,  we also put $\hat q:=\qQ(\hat X)$.

\begin{slemma}[{\cite[Lemma~4.12]{P:2010:QFano}}]
\label{lemma:sl:torsion} 
If $\Clt{X}=0$, then $\Clt{\hat{X}}$ is a cyclic group of order $d/e$. 
If $\Clt{\hat X}=0$, then $\Clt{X}$ is a cyclic group of order $e/d$. 
\end{slemma}

\begin{sremark}
\label{rem:case-bir}
In the above notation $s>0$ and $s_k=0$ if and only if 
$\dim\MMM_k=0$ and the unique element $M_k\in\MMM_k$ 
coincides with the $\bar f$-exceptional divisor $\bar{F}$.
\end{sremark}

\subsection{Fibrations}
\label{case-nonbir}
Assume that $\bar{f}$ is a fibration. Then we denote by
$\bar{F}$ a general geometric fiber.
Then $\bar{F}$ is either a smooth rational curve or a smooth del Pezzo surface.
The image of the restriction map $\Cl(\bar{X})\to\Pic (\bar{F})$ is isomorphic 
to $\ZZ$.
Let $\Xi$ be its ample generator.
As above, we can write
\[
-K_{\bar{X}} |_{\bar{F}}=-K_{\bar{F}}\sim\hat{q}\Xi,
\qquad
\bar{E} 
|_{\bar{F}}\sim e\Xi,
\qquad
\bar\MMM |_{\bar{F}}\sim s\Xi,
\qquad
\bar\MMM_k |_{\bar{F}}\sim s_k\Xi,
\]
where $\hat{q}\in\ZZ_{>0}$, and $e,\, s,\, s_k\in\ZZ_{\ge0}$.
Note that $\bar E$ is $\bar f$-ample \cite[Claim~4.6]{P:2010:QFano}, hence $e>0$.

\begin{sremark}[{\cite[Lemma~5.1]{P:QFano-rat1}}]
\label{rem:case-nonbir}
If in the above assumption   $X$ is not 
rational, then
$\hat q=1$.
\end{sremark}

\subsection{Numerical constraints}
Taking the expressions \eqref{equation-1} into account, one can easily deduce the following relations (see \cite[(4.2.2)]{P:QFano-rat1}). 
They will be frequently used  below.
\begin{eqnarray}
\label{eq:main}
n\hat{q}&=&q s+(q\beta-n\alpha) e,
\\[3pt]
\label{eq:main1}
k\hat{q}&=&q s_k+(q\beta_k-k\alpha) e,
\end{eqnarray}
where $q\beta-n\alpha>0$ (see~\eqref{eq:ct:beta-alpha}). If furthermore $\qQ(X)=\qW(X)$, then 
$q\beta-n\alpha$ is a positive integer.

Now assume that the morphism $\bar{f}$ is birational.
Similar to~\eqref {equation-1} we can write
\[
K_{\bar{X}}\qq\bar{f}^*K_{\hat{X}}+b\bar F,\quad
\bar\MMM_k\qq\bar{f}^*\hat\MMM_k-\gamma_k\bar F,\quad
\bar E\qq\bar{f}^*\hat{E}-\delta\bar F.
\]
If $\Cl(\hat{X})\simeq\ZZ$, then similar to~\eqref{eq:main1}
we have the following relations (see \cite[2.6]{P:2016:QFano7}): 
\begin{equation} 
\label{eq:-b-gamma-delta-1}
\begin{array}{lll}
be &=&\hat{q}\delta-q,
\\[2pt]
e\gamma_k &=& s_k\delta-k.
\end{array}
\end{equation}

\section{$\QQ$-Fano threefolds with torsions in $\Cl(X)$}
\label{sect:tor}
In this section we collect facts on $\QQ$-Fano threefolds having nontrivial
torsions in $\Cl(X)$.
\begin{lemma}[{\cite[Proposition~2.9]{P:2010:QFano}}]
\label{lemma:TOR}
Let $X$ be a Fano threefold with terminal singularities and let $T$ be an 
$n$-torsion element in the Weil divisor class group. Let $\B(X)^T$ be the 
collection of points $P\in\B(X)$ such that $T$ is 
not 
Cartier at $P$.
Assume furthermore that $n$ is prime. Then
\begin{enumerate}
\item
$n\in\{2, 3, 5, 7\}$.
\item
If $n = 7$, then $\B(X)^T = (7, 7, 7)$.
\item
If $n = 5$, then $\B(X)^T = (5, 5, 5, 5)$, $(10, 5, 5)$, or $(10, 10)$.
\item
\label{lemma:TOR3}
If $n = 3$, then $\sum_{P\in\B(X)^T}\ind(X,P) = 18$.
\item
If $n = 2$, then $\sum_{P\in\B(X)^T}\ind(X,P) = 16$.
\end{enumerate}
\end{lemma}

\begin{lemma}[{\cite[Lemma~3.2]{P:2010:QFano}}, {\cite[Lemma~2.3]{P:2019:rat:Q-Fano}}]
\label{lemma:qQ=qW}
Let $X$ be a $\QQ$-Fano threefold, let $q:=\qQ(X)$, and let $r:=\ind(X)$ be the 
global Gorenstein index of $X$.
Then the equality $\qQ(X)=\qW(X)$ holds if and only if $q$ and~$r$ are coprime, 
and if and only if
the order of $K_X+qA_X$ in the group $\Cl(X)$ is prime to $q$.
\end{lemma}

\begin{lemma}[{\cite[Proposition~3.3]{P:2019:rat:Q-Fano}}]
\label{lemma:qQ=qW:34}
Let $X$ be a $\QQ$-Fano threefold with $\qQ(X)\neq\qW(X)$. 
Then $\qQ(X)\le 4$.
\end{lemma}

\begin{lemma}
\label{lemma:q=3:tor}
Let $X$ be a $\QQ$-Fano threefold with $\qQ(X)=3$, $\Clt{X}\simeq\ZZ/3\ZZ$, and 
$A_X^3=1/4$.
Then $X$ is rational.
\end{lemma}
The existence of varieties satisfying the conditions of the lemma is not known.
It is expected that they are quotients of codimension-$4$ \ $\QQ$-Fano threefolds
\# 41218 in \cite{GRD}.

\begin{proof}
In this case $\qQ(X)\neq\qW(X)$ by Lemma~\ref{lemma:qQ=qW}, hence $\qW(X)=1$ 
and by \cite[Table~3]{P:QFano-rat1} we have 
\[
\B(X)=(2, 3^2, 12)\quad \text{and}\quad A^3_X= 1/4.
\]
Put 
\[
T:=K_X+3A_X. 
\]
Then $T$ is a generator of the group $\Clt{X}\simeq\ZZ/3\ZZ$ and it is not 
Cartier at
any point whose index is divisible by $3$ (see 
Lemma~\ref{lemma:TOR}\ref{lemma:TOR3}).
Let $P\in X$ be the point of index~$12$. Then by Lemma~\ref{lemma:K-index}
we have $T\overset{P}{\sim} -lK_X$, where
$0<l<12$ and $3l\equiv 0\mod 12$, hence $l=4$ or $8$. On the other hand,
$3A_X=T-K_X\overset{P}{\sim} (l+1)(-K_X)$. Hence $4(l+1)\equiv 0\mod 12$ and so
$l=8$. Thus $3A_X\overset{P}{\sim} 9(-K_X)$ and so $A_X\overset{P}{\sim} 
a(-K_X)$, where $a\in\{3,\, 7,\, 11\}$.
Replacing $A_X$ with $A_X\pm T$, if necessary, we may assume that 
$A_X\overset{P}{\sim} 11(-K_X)$ and then $2A_X\overset{P}{\sim} 10(-K_X)$. 

By the orbifold Riemann-Roch (see \cite[Proposition~3.2]{P:QFano-rat1}) we have
\[
\dim|2A_X |=\dim|2A_X+T|=\dim|2A_X+2 T|=1.
\]
Assume that $X$ is not rational.
Apply the construction~\eqref{diagram-main} with $\MMM=|2A_X|$.
Then $\beta\ge 10\alpha$ by~\eqref{eq:ct:beta-alpha}.
Hence the relation~\eqref{eq:main}  has the form
\[
2\hat{q}=3s+(3\beta- 2\alpha) e\ge 3s+28\alpha e.
\]
Since $\alpha\ge 1/12$, the above relation implies that $\hat q>1$.
Then $\bar f$ is birational (see Remark~\ref{rem:case-nonbir}), $s>0$, and 
$\hat q\ge 3$.
On the other hand, $\hat q\le 7$ by Theorem~\ref{thm0}.
If $\alpha\ge 1/3$, then $\alpha=1/3$, $\hat q=7$, $s=1$, hence $\p_1(\hat X)\ge 2$.
Again by Theorem~\ref{thm0} the variety $X$ is rational in this case, a 
contradiction.

Therefore, $\alpha=1/12$ and $f(E)=P$.
If $\hat q=3$, then $s=e=1$ and
$\beta=19/18\notin\frac{1}{12}\ZZ$, a contradiction.
Therefore, $\hat q\ge 4$. Then $s\ge 2$ by Theorem~\ref{thm0}, hence
$\hat q\ge 5$. If $\hat q=5$, then $s=2$, $e=1$, $3\beta- 2\alpha=4$,
and again $\beta=25/18\notin\frac{1}{12}\ZZ$, a contradiction.
Thus, $\hat q\ge 6$ and $s\ge 4$ by Theorem~\ref{thm0},
hence $\hat q\ge 8$. The contradiction concludes the proof.
\end{proof}

\begin{lemma}
\label{prop:q=5:tor}
Let $X$ be a $\QQ$-Fano threefold with $\qQ(X)=5$ and $A_X^3=1/18$.
Assume that $X$ is not rational. Then $\Clt{X}=0$.
\end{lemma}

Note that $\QQ$-Fano threefolds with $\qQ(X)=5$ and $A_X^3=1/18$ appear in 
\cite{GRD} as \#41414. 
Such varieties do exist: see Example~\ref{ex:tor},~\ref{tab:h:q=5t3:d=1/18}.
However, the existence of such varieties with torsion free group $\Cl(X)$ is 
not known.

\begin{proof}
By Lemma~\ref{lemma:qQ=qW:34} we have $\qQ(X)=\qW(X)$ and applying the computer 
search (or~\cite{GRD})
we obtain 
\[
\B(X)=(2, 9^2),\quad\dim|A_X|=\dim| 2 A_X|=\dim| 3 A_X|= 0,\quad\dim| 4 
A_X|= 1.
\]
Moreover, if $\Clt{X}\neq 0$, then $\Clt{X}\simeq\ZZ/3\ZZ$ and for a generator 
$T\in\Clt{X}$ we have
\begin{equation}
\label{eq:tor:q=5}
\begin{array}{lll}
\phantom{\dim\, }|A_X+T|&=\dim|A_X+2 T|&=\varnothing, 
\\
\dim| 2 A_X+T|&=\dim| 2 A_X+2 T|&= 0,
\\
\dim| 3 A_X+T|&=\dim| 3 A_X+2 T|&= 1. 
\end{array}
\end{equation} 
This implies that 
$|4A_X|$ is a pencil without fixed components (regardless of torsions in 
$\Cl(X)$).
Apply the construction~\eqref{diagram-main} with $\MMM=|4A_X|$.
Note that $\alpha\in\ZZ\cup\{1/2,\, 1/9,\, 2/9\}$ (see 
Lemma~\ref{lemma-discrepancies})
and $\beta\ge 8\alpha$. The relations~\eqref{eq:main} and~\eqref{eq:main1} with $k=1$ have 
the form
\begin{equation}
\label{eq:m:TOR:q=5}
\begin{array}{rcl}
4\hat{q}&=&5s+(5\beta- 4\alpha) e\ge 5s+36\alpha e\ge 5s+4 e,
\\[3pt]
\hat{q}&=&5s_1+(5\beta_1- \alpha) e.
\end{array}
\end{equation} 
We claim that only the following possibilities can occur:\footnote{We will show below
in Propositions~\ref{prop:q6} and~\ref{prop:q7a} that in the 
cases~\ref{prop:q=5:tor-d} 
and~\ref{prop:q=5:tor-c} the variety $X$ is rational.}
\begin{enumerate}
\renewcommand{\theenumi}{\rm(\alph{enumi})}
\renewcommand{\labelenumi}{\rm(\alph{enumi})}
\item 
\label{prop:q=5:tor-a}
$\hat q=1$, $\alpha=1/9$, $e=1$, $s =s_1=0$;
\item
\label{prop:q=5:tor-d}
$\hat q= 6$, $\alpha= 1/9$, $e= 1$, $s = 4$, $s_1\in \{0,\, 1\}$;
\item 
\label{prop:q=5:tor-c}
$\hat q=7$, $e\alpha= 2/9$, $e \in \{1,\, 2\}$, $s=4$, $s_1\in \{0,\, 1\}$, 
$e-s_1\le 1$.
\end{enumerate}
If $\hat q=1$, then $s=0$ and equalities hold in \eqref{eq:m:TOR:q=5}. This 
is the case~\ref{prop:q=5:tor-a}.
Let $\hat q>1$, then $\bar f$ is birational, 
$s>0$, and $\hat q\ge 3$. As in the proof of Lemma~\ref{lemma:q=3:tor}
we obtain $\hat q=6$ or $7$. In this case $s\ge 4$, then by 
\eqref{eq:m:TOR:q=5}
we have
$s=4$ and there are only the possibilities~\ref{prop:q=5:tor-d} 
and~\ref{prop:q=5:tor-c}.

Now assume that
$\Clt{X}\neq 0$.
Since $\alpha=1/9$ or $2/9$, $P:=f(E)$ is a point of index~$9$.
Then $\Cl(X,P)\simeq\ZZ/9\ZZ$ (see Lemma~\ref{lemma:K-index}) and elements $T$ 
and~$3A_X$ have order $3$ in this group.
Hence, replacing $T$ with $-T$ is necessary, we may assume that $3A+T$ is 
Cartier at $P$.
It follows from~\eqref{eq:tor:q=5} that $\MMM'=|3A+T|$ a pencil without 
fixed components.
For this linear system similar to~\eqref{eq:main1} we have 
\[
3\hat{q}=5s_3'+(5\beta_3'- 3\alpha) e\ge 5s_3'-1,
\]
because $\beta'_3$ is a nonnegative integer and~$3\alpha e<1$.
If $\hat q=1$, then $s_3'=0$, $e=1$, and $\alpha=1/9$. Hence 
$5\beta_3' = 10/3\notin\ZZ$, a contradiction.
Therefore, $\hat q\in \{6,\, 7\}$ and $s_3'\ge 4$ by Theorem~\ref{thm0}.
In this case, $s_3'= 4$ and the number $5\beta_3'= 3\alpha+1/e$
is not an integer in both cases~\ref{prop:q=5:tor-d} and~\ref{prop:q=5:tor-c}.
The contradiction completes the proof.
\end{proof}

\begin{scorollary}
\label{cor:q=3:tor}
Let $X$ be a $\QQ$-Fano threefold such that  $\qQ(X)\ge 3$, $\Clt{X}\neq 0$,
and either $|\Clt{X}|\ge 3$ or  $\p_1(X)\ge 2$.
Then $X$ is rational except, possibly, for the following case:
\begin{enumerate}
\item[*)]
\label{prop:q=3:tor3}
$\qQ(X)=3$, $\qW(X)=1$, $\Clt{X}\simeq\ZZ/3\ZZ$, $\B(X)=(3^4, 5, 6) $, $A_X^3= 
1/10$.
\end{enumerate}
\end{scorollary}
Note however, that the existence of $\QQ$-Fano threefolds in~*) is not 
known.

\begin{proof}
If $\p_1(X)\ge 2$, then $X$ is rational by Theorem~\ref{thm0} and \cite[Proposition~6.4]{P:QFano-rat1}.
Thus we  may assume that $|\Clt{X}|\ge 3$.
Consider the case  $\qQ(X)=\qW(X)$. By \cite[Proposition~3.4]{P:QFano-rat1} we have 
$\Clt{X}\simeq\ZZ/3\ZZ$ and $\qQ(X)\neq 4$. 
By Lemma~\ref{lemma:qQ=qW} the $\QQ$-Fano index~$\qQ(X)$
must be prime to the order of $\Clt{X}$
and so $\qQ(X)\ge 5$. Then by Proposition~\ref{prop:tor}
the variety $X$ is described by \ref{tab:h:q=5t3:d=1/18} in Table~\ref{tab:2}, 
i.e. 
we are in the situation of Lemma~\ref{prop:q=5:tor}. Then $\Clt{X}= 0$ which contradicts our 
assumption.

Therefore, $\qQ(X)\neq\qW(X)$. Then by \cite[Proposition~3.2]{P:QFano-rat1}
we have $\qQ(X)=3$ and $\Clt{X}\simeq\ZZ/3\ZZ$ 
and $X$ is one of the varieties $1^o$, $2^o$, $3^o$ from 
\cite[Table~3]{P:QFano-rat1}.
The variety $1^o$ (resp. $3^o$) is rational by 
\cite[Proposition~3.3]{P:QFano-rat1}
(resp. by Lemma~\ref{lemma:q=3:tor}). Then we are left with the case *).
\end{proof}

\begin{proposition}
\label{prop:t}
Let $X$ be a $\QQ$-Fano threefold with $\qQ(X)\ge 5$ and $\Clt{X}\neq 0$.
Then $X$ is rational.
\end{proposition}

First, we prove the following lemma.
\begin{lemma}
\label{prop:q=7:tor}
Let $X$ be a $\QQ$-Fano threefold with $\qQ(X)=7$ and $A_X^3=1/30$ 
\textup(case~\xref{Case:q=7:B=2-6-10} of Table~\xref{tab:1}\textup).
Then $X$ is rational.
\end{lemma}

\begin{proof}
In this case $\B(X)=(2,6,10)$ 
and $X$ has only
cyclic quotient singularities.
Note that the linear system $|6A_X|$ has no fixed components and $\dim|6A_X|=4$.
Apply the construction~\eqref{diagram-main} with $\MMM=|6A_X|$.
If $P$ is the point of index~$10$, then $\MMM\overset{P}{\sim} 8(-K_X)$ and so 
$\beta\ge8\alpha$
by~\eqref{eq:ct:beta-alpha}.
Since $6M_1\in\MMM$, we have $6\beta_1\ge\beta$.
Hence $\beta_1\ge\frac 16\beta\ge\frac 43\alpha$. 
The relation~\eqref{eq:main}  can be written as follows:
\begin{equation}
\label{eq:main:q=7a:6a}
6\hat{q}=7 s+(7\beta-6\alpha) e\ge 7 s+50\alpha e\ge 7 s+5 e.
\end{equation}
If $\alpha\ge 1/2$, then $\hat q\ge 5$, hence $s\ge 1$,
$\hat q\ge 6$, and $s\le 2$.
This contradicts Theorem~\ref{thm0}. Therefore, $\alpha=1/r$, where $r=6$ or 
$10$.

Assume that $r=6$. Then~\eqref{eq:main1} with $k=1$ gives us 
$6\beta_1=(6\hat{q}+e)/(7e)$.
Since $6\beta_1$ is an integer and $\hat q\le 7$, we obtain $\hat q=e$ and 
$\beta_1=1/6$.
This contradicts our inequality $\beta_1>\alpha=1/r$.

Therefore, $r=10$. As above we obtain $s_1=0$ and $\hat q=2e$.
Then $3\ge e\ge s$ by~\eqref{eq:main:q=7a:6a}.
If $e=1$ or $3$, we get a contradiction by Theorem~\ref{thm0} because 
$\dim\MMM=4$.
Hence $e=2$. Then $\hat q=4$ and $s=2$. In this case $\Clt{X}\neq 0$ by 
Lemma~\ref{lemma:sl:torsion} because $s_1=0$ and $e=2$. 
Moreover, $\Clt{X}\simeq \ZZ/2\ZZ$ by Lemma~\ref{lemma:TOR}.
Let $T\in\Cl(X)$ be the $2$-torsion element.
Then 
$\MMM':=|3A_X+T|$ is a pencil without fixed components (see 
~\eqref{eq:h::q=7:B=2-6-10}). We have
$\tilde\MMM_3'=f^*\MMM_3' -\beta_3' E$, where $\beta_3'=l/10$, $l$ is an 
integer and $l>0$ because $3A_X+T$ is not 
Cartier at $P$. 
Similar to~\eqref{eq:main1} we have
\begin{equation*}
%\label{eq:q=6}
\textstyle
12=3\hat{q}=7 s_3'+(7\beta_3'-3\alpha) e= 7 s_3'+\frac{7 l-3}{5},\qquad 5 
s_3'+l=9.
\end{equation*}
Thus $s_3'=1$ and so $\p_1(\hat X)\ge 2$. This contradicts Theorem~\ref{thm0}.
\end{proof}

\begin{proof}[Proof of Proposition~\xref{prop:t}]
Assume that $X$ is not rational.
By Proposition~\ref{prop:tor} we have $\qQ(X)=\qW(X)$, $\qQ(X)\neq 6$, and 
$\p_3(X)\ge 2$
in the case $\qQ(X)=7$. 
The latter is impossible by Theorem~\ref{thm0}. Hence $\qQ(X)=5$ and
we have only 
the cases \ref{tor:q=5:1/12} and \ref{tor:q=5:1/28} of Table~\ref{tab:2}
(again by Theorem~\ref{thm0}). Then $\B(X)=(4^2,12)$ and $(2,4,14)$ in these 
cases, respectively.
In both cases $\Clt{X}\simeq\ZZ/2\ZZ$.
Let $T$ be the generator of $\Clt{X}$. 
Apply the construction~\eqref{diagram-main} with $\MMM=|4A_X|$.
By Proposition~\ref{prop:tor} we have $\dim \MMM=3$ in the case 
\ref{tor:q=5:1/12} and $\dim \MMM=1$ in the case \ref{tor:q=5:1/28}.
If $P$ is the point of index~$12$ (resp. index~$14$), then $\MMM 
\overset{P}\sim 8(-K_X)$ (resp. $\MMM \overset{P}\sim 12(-K_X)$) and so 
$\beta\ge 8\alpha$
by~\eqref{eq:ct:beta-alpha}.
In both cases we have (see~\eqref{eq:main}):
\begin{equation}
\label{eq:q=5:tor}
4\hat{q}=5s+(5\beta- 4\alpha) e\ge 5s+36\alpha e.
\end{equation} 
First, assume that $\alpha\ge 1/4$.
Then $\hat q>1$, hence $s>0$, $\hat q\ge 4$, $s\ge 2$ by Theorem~\ref{thm0} 
and so $\hat q\ge 5$. 
If $\hat q\ge 6$, then $s\ge 4$ by Theorem~\ref{thm0} and so $\hat q>7$, a 
contradiction.
Therefore, $\hat q=5$. In this case, $s=2$ and $e=1$.
Moreover, $\alpha\le 5/18$, hence $f(E)$ is a point of index~$4$, $\alpha=1/4$ 
and $\beta=11/5$.
On the other hand, $\beta$ must be an integer
because $\MMM$ is Cartier at $f(E)$.
The contradiction shows that $\alpha< 1/4$.

Then $f(E)=P$ is a cyclic quotient singularity of index~$r=12$ (resp. $r=14$) 
in the case 
\ref{tor:q=5:1/12} (resp. \ref{tor:q=5:1/28}). 
In particular, $\alpha=1/r$ and $\beta=l/r$, where $l$ is an integer.
The relation \eqref{eq:q=5:tor} can be rewritten as follows:
\begin{equation}
\label{eq:q=5:tor-new}
4\hat{q}r= 5sr+(5 l- 4) e\ge 5sr+36 e.
\end{equation}
Consider the case $\hat q=1$.
Then $e=1$, $s=0$, and $r=14$, hence we are in the case \ref{tor:q=5:1/28}
and $\bar f$ is a fibration. Moreover, the linear system $\bar \MMM$ is $\bar 
f$-vertical,
i.e. $\bar \MMM=\bar f^*\NNN$, where $\NNN$ is a complete linear system on 
$\hat X$ such that $\dim \NNN=1$.
Since $4M_1\in \MMM$, the divisor $\bar M_1$ is $\bar f$-vertical as well. If 
$\hat X$ is a surface, 
then 
and $\bar M_1=\bar f^* N_1$, where $N_1$ is an effective irreducible divisor on 
$\hat X$ such that $\NNN\sim 4N_1$.
But then
$\dim \NNN\ge 3$ by Corollary~\ref{cor:base-surface}, a contradiction.
Therefore, $\hat X\simeq \PP^1$, $|\NNN|=|\OOO_{\PP^1}(1)|$, and $N_1$ is a 
multiple fiber of $\bar f$.
Now, consider the linear system $\MMM':=|4A_X+T|$. It is also a pencil without 
fixed components.
Similar to \eqref{eq:q=5:tor} we have the following relation
\begin{equation*}
% \label{eq:q=5:tor}
\textstyle
4=4\hat{q}=5s_4'+(5\beta_4'- 4\alpha) e=5s_4'+5\beta_4'- \frac 27.
\end{equation*} 
Then $s_4'=0$. As above, $\bar \MMM'=\bar f^*\NNN'$, where $\NNN'$ 
is a complete linear system on $\hat X\simeq \PP^1$ such that $\dim \NNN'=1$.
But then we must have $\NNN'=\NNN$ and so $\MMM'=\MMM$, a contradiction.
Therefore, $\hat q>1$, $\bar f$ is a birational contraction, and $s>0$. 

Assume that $r=12$ (i.e. we are in the case \ref{tor:q=5:1/12}).
Consider the linear system $\MMM':=|3A_X+T|$. 
As above, it is a pencil without fixed components.
Note that $T\overset{P}{\sim} 6(-K_X)$ and $\MMM'\overset{P}{\sim} 9(-K_X)$.
Hence $\beta_3'\equiv 9\alpha =3/4 \mod \ZZ$ and we can write $\beta_3'=3/4+l'$ 
with $l'\in \ZZ_{\ge0}$.
Similar to \eqref{eq:q=5:tor} we have 
\begin{equation}
\label{eq:q=5:to:s1}
\textstyle
3\hat{q}=5s_3'+(5\beta_3'- 3\alpha) e= 5s_3' +\left(\frac 72 +5l'\right) e.
\end{equation} 
Now, \eqref{eq:q=5:tor-new} implies that $12\hat{q}+e\equiv 0\mod 5$ and 
$4\hat{q}> 3e$. This gives us only the following possibilities: 
$(\hat{q}, e)=(2, 1),
(4, 2),
(5, 5),
(6, 3),
(7, 1),
(7, 6)$.
In the cases $(\hat{q}, e)=(2, 1)$ and $(5, 5)$
we have $s=1$ (again by \eqref{eq:q=5:tor-new}). 
This contradicts 
Theorem~\ref{thm0} because $\dim \MMM=3$. The same arguments work in the cases 
$(\hat{q}, e)=(6, 3)$ and $(7, 6)$. In the case $(\hat{q}, e)=(4, 2)$ 
we have $s_3'=1$ (see \eqref{eq:q=5:to:s1}), hence $\p_1(\hat X)\ge 2$ which 
again contradicts Theorem~\ref{thm0}.
Finally, in the case $(\hat{q}, e)=(7, 6)$ we again get a contradiction by 
\eqref{eq:q=5:to:s1}.

Therefore, we are in the case \ref{tor:q=5:1/28} and $r=14$.
Then $\beta\ge 12\alpha$ and the inequalities in \eqref{eq:q=5:tor} and 
\eqref{eq:q=5:tor-new} can be improved:
\begin{eqnarray*}
4\hat{q}&=&5s+(5\beta- 4\alpha) e\ge 5s+56\alpha e,
\\
56\hat{q}&=&70s+(5 l- 4) e\ge 70s+56 e.
\end{eqnarray*}
As above $14\hat{q}+e\equiv 0\mod 5$ and $\hat{q}\ge e$. Taking into account 
that
$s>0$ whenever $\hat q>1$, we obtain only three possibilities: $(\hat{q}, 
e)=(1,1)$, $(6, 1)$, and $(7, 2)$. 
Now, consider the linear system $\MMM':=|4A_X+T|$. It is a pencil without fixed 
components.
Note that $T\overset{P}{\sim} 7(-K_X)$ and $\MMM'\overset{P}{\sim} 5(-K_X)$.
Hence $\beta_4'\equiv 5\alpha =5/14 \mod \ZZ$ and we can write 
$\beta_4'=5/14+l'$ with $l'\in \ZZ_{\ge0}$.
Similar to \eqref{eq:q=5:tor} we have the following relation
\begin{equation*}
% \label{eq:q=5:tor}
\textstyle
4\hat{q}=5s_4'+(5\beta_4'- 4\alpha) e=5s_4'+ \frac32 e+ 5l'e.
\end{equation*} 
Then $e$ must be even, so $(\hat{q}, e)=(7, 2)$. If $\Clt{X}\neq 0$, then 
$\p_3(X)\ge 2$
by Proposition~\ref{prop:tor}, so $\hat X$ is rational by Theorem~\ref{thm0}.
Therefore, the group $\Cl(\hat X)$ is torsion free.
Then by Lemma~\ref{lemma:sl:torsion} \ $d=1$, i.e. $s_1=0$, and
from~\eqref{eq:main1} we obtain $\beta_1=5/7$.
On the other hand, $A_X\overset{P}{\sim} 3(-K_X)$ and so $\beta_1\equiv 
3\alpha\mod \ZZ$,
a contradiction.
\end{proof}

\begin{examples}
\label{ex:tor}
Here are some examples of $\QQ$-Fano threefolds with $\qQ(X)\ge 5$ and 
$\Clt{X}\neq 0$
(cf. Table~\ref{tab:2}).
\begin{enumerate}
\item[\ref{tab:h:q=7t:d=1/24}]:
$X_6\subset \PP(1,2,3^2,4)/\mumu_2(0,1,0,1,1)$,
\item[\ref{tab:h:q=7t:d=1/30}]:
$X_8\subset \PP(1,2,3,4,5)/\mumu_2(0,1,1,1,1)$,
\item[\ref{tab:h:q=5t3:d=1/18}]:
$X_6\subset \PP(1,2^2,3^2)/\mumu_3(0,1,2,1,2)$,
\item[\ref{tab:h:q=5t2:d=1/6}]:
$X_4\subset \PP(1^2,2^2,3)/\mumu_2(0,1,1,1,0)$,
\item[\ref{tab:h:q=5t2:d=1/8}]:
$X_6\subset \PP(1^2,2,3,4)/\mumu_2(0,1,1,1,1)$.
\end{enumerate}
The existence of $\QQ$-Fano threefolds of types
\ref{tor:q=5:1/12} or
\ref{tor:q=5:1/28}
of Table~\ref{tab:2} is not known.
\end{examples}

\section{A family of nonrational $\QQ$-Fano threefolds of index~$5$}
\label{sect:q=5}
In this  section we discuss the birational properties of 
an ``extremal'' $\QQ$-Fano threefold of index~$5$, see 
Theorem~\ref{thm0}\ref{thm0:5}.
\begin{setup}
\label{setup:q=5}
Let $X$ be a \textit{nonrational} $\QQ$-Fano threefold with
\begin{equation*}
%\label{eq:ex:q=5}
\qQ(X)=5,\qquad A_X^3=1/12,\qquad\B(X)=(2^2, 3, 4).
\end{equation*}
In this situation the group $\Cl(X)$ is torsion free (see 
Lemma~\ref{lemma:TOR}), and
\begin{equation}
\label{eq:QFanoq=5:h}
\dim|kA_X|=k-1 \quad \text{for $k=1,2,3$},\qquad \dim|4A_X|=4,\qquad \dim 
|5A_X|=6.
\end{equation} 

This variety has id number \#41422 in \cite{GRD}.
A quasi-smooth weighted hypersurface $X_{10}\subset\PP(1,2,3,4,5)$ satisfies 
these conditions (see Corollary~\ref{cor:q=5:non-rat} below).
Conversely, we  will show below in Theorem~\ref{thm:q5} that any 
\textit{nonrational} $\QQ$-Fano threefold with $\qQ(X)=5$ and $\B(X)=(2^2, 3,
4)$
is isomorphic to a hypersurface $X_{10}\subset\PP(1,2,3,4,5)$.
\end{setup}

Since $\B(X)=(2^2, 3, 4)$, the collection of non-Gorenstein singularities of $X$ is as follows:
\begin{itemize}
\item
a unique point $P_3$ of index~$3$ that is a cyclic quotient,
\item
a unique point $P_4$ of index~$4$ that is either a cyclic quotient or a
singularity of type~\typeA{cAx}{4},
\item
index~$2$ singularities $P_2^{(1)}$,\dots,$P_2^{(m)}$, where $0\le m\le 2$ and 
$\aw (X,P_4)+\sum\aw(X,P_2^{(i)})=3$.
\end{itemize}

\begin{proposition}
\label{prop:q=5:link2}
Let $X$ be a $\QQ$-Fano threefold satisfying~\xref{setup:q=5}.
Then there exists a Sarkisov link of the form~\eqref{diagram-main} such that
\begin{enumerate}
\item
$f$ is an extremal blowup of the index-$4$ point $P_4$,
\item
$\bar f$ is a $\QQ$-conic bundle.
\end{enumerate}
\end{proposition}
More precise description of this link will be given in 
Proposition~\ref{prop:q=5:link-cb}.
\begin{proof}
It follows from~\eqref{eq:QFanoq=5:h} that $|2A_X|$ and $|3A_X|$ are linear 
systems without fixed components and $|A_X|\neq\varnothing$.
Apply the construction~\eqref{diagram-main} with $\MMM=|3A_X|$.
If $P$ is the point of index~$4$, then $\MMM\overset{P}{\sim} 3(-K_X)$ and so 
$\beta\ge 3\alpha$
by~\eqref{eq:ct:beta-alpha}.
The relation~\eqref{eq:main} in this case has the form
\begin{equation}
\label{eq:q=5:m}
3\hat{q}=5 s+(5\beta-3\alpha) e\ge 5 s+12\alpha e\ge 5 s+3 e.
\end{equation}
Assume that $s>0$. Then $\hat q\ge 3$ and the contraction $\bar f$ is 
birational by Remark~\ref{rem:case-nonbir}. 
Since $\dim\MMM=2$ and $\hat X$ is not rational,
$s\ge 2$ by Theorem~\ref{thm0}. Hence $\hat q\ge 5$ by~\eqref{eq:q=5:m} and 
$s\ge 3$ again by Theorem~\ref{thm0}. Then we conclude successively $\hat q\ge 6$, $s\ge 4$, 
and $\hat q>7$, a contradiction.

Thus $s=0$. Then $\bar f$ is a fibration by Remark~\ref{rem:case-bir} and 
$\hat q=1$ by Remark~\ref{rem:case-nonbir}. From~\eqref{eq:q=5:m} we conclude 
that $s=0$ and $e=1$
and from~\eqref{eq:main1} with $k=2$ we obtain $s_2=0$.
Thus $\bar\MMM_2=\bar f^*\hat\MMM_2$ and $\bar\MMM=\bar f^*\hat\MMM$, where
$\hat\MMM_2$ and $\hat\MMM$ are complete linear systems with 
$\dim\hat\MMM_2=1$ and $\dim\hat\MMM=2$.
If $\hat X\simeq\PP^1$, then $\hat\MMM_2=|\OOO_{\PP^1}(1)|$ and 
$\hat\MMM=|\OOO_{\PP^1}(2)|$,
hence $\hat\MMM\sim 2\hat\MMM_2$. Pushing down this relation to $X$ we get a 
contradiction.
Therefore $\hat X$ is a surface and $\bar f$ is a $\QQ$-conic bundle. 
\end{proof}

\begin{scorollary}
\label{cor:q=5:P(1,2,3)}
In the above notation we have $\bar M_1\sim\bar f^* A_{\hat X}$, 
$\bar\MMM_2=\bar f^* |2A_{\hat X}|$, and $\bar\MMM_3=\bar f^* |3A_{\hat X}|$.
\end{scorollary}

\begin{scorollary}
\label{cor:q=5:Bs4}
$P_4\notin\Bs|4A_X|$.
\end{scorollary}

\begin{proof}
Since $\dim|4A_X|>\dim|3A_X|$, the linear system $|4A_X|$ has no fixed 
components.
Assume that $P_4\in\Bs|4A_X|$. Then $\beta_4>0$.
From~\eqref{eq:main1} for $k=4$ we obtain $s_4 =0$.
Then as above $\bar\MMM_4=\bar f^* |m A_{\hat X}|$ for some $m$.
Pushing down this relation to $X$ we obtain $m=4$ and so
$\dim|4 A_{\hat X}|=\dim\bar\MMM_4=4$. This contradicts 
Corollary~\ref{cor:base-surface}.
Therefore, $s_4>0$ and $\beta_4 =0$.
\end{proof}

\begin{proposition}
\label{prop:q=5o:main}
Let $X$ be a $\QQ$-Fano threefold satisfying~\xref{setup:q=5}.
Let $P_3\in X$ be the point of index~$3$ and let $f:\tilde X\to X$
be the Kawamata blowup of $P_3$. Then $f$ can be completed to
the following Sarkisov link
\begin{equation}
\label{eq:sl:q=5}
\vcenter{
\xymatrix@R=1em{
&\tilde X\ar@/_0.4em/[dl]_{f}\ar@/^0.4em/[dr]^{\bar f}&
\\
X&&\hat X
}}
\end{equation}
where $\hat X$ is a $\QQ$-Fano hypersurface of degree $6$ in $\PP(1^2,2^2,3)$ 
and $\bar f$ contracts
a divisor $\tilde M_1$ to a smooth rational curve $\hat\Upsilon\subset\hat X$
such that $A_{\hat X}\cdot\hat\Upsilon=1/2$.
\end{proposition}

\begin{proof}
First, we claim that  the linear system 
$\tilde\MMM_4$ is nef.
Indeed, assume that $\tilde\MMM_4\cdot\tilde \Gamma<0$ for some irreducible curve $\tilde \Gamma$.
Then $\tilde \Gamma\subset\Bs\tilde\MMM_4$ and $\tilde \Gamma\cap E\neq\varnothing$, hence
$E\cap\Bs\tilde\MMM_4\neq\varnothing$ and $P_4\in\Bs\MMM_4$. This contradicts 
Corollary~\ref{cor:q=5:Bs4}.
Therefore, $\tilde \Gamma\subset E$. On the other hand, $\uprho (\tilde X/X)=1$, 
hence
$\tilde\MMM_4$ is $f$-ample, a contradiction.
Thus $\tilde\MMM_4$ is nef.

Since $E^3=9/2$, we have
\begin{align*}
\tilde\MMM_4^2\cdot\tilde M_1&=\MMM_4^2\cdot M_1-\beta_4^2\beta_1 
E^3=
\textstyle
\frac 43 -\frac 92\beta_4^2\beta_1,
\\\textstyle
\tilde\MMM_4\cdot\tilde M_1^2&=\MMM_4\cdot M_1^2-\beta_4\beta_1^2 
E^3=
\textstyle
\frac 13 -\frac 92\beta_4\beta_1^2.
\end{align*}
Since $\MMM_4 \overset{P_3}\sim M_1\overset{P_3}\sim -2K_X$, we have 
$\beta_4\equiv\beta_1\equiv 2\alpha=2/3\mod 3$.
Therefore, $\beta_4,\,\beta_1\ge 2/3$.
Since the linear system $\tilde\MMM_4$ is nef, we have 
$\tilde\MMM_4^2\cdot\tilde M_1\ge 0$, hence
$\beta_4=\beta_1=2/3$.
Thus
\[\textstyle
\tilde\MMM_4\qq f^*\MMM_4-\frac 23 E,\qquad\tilde M_1\qq f^*M_1-\frac 23 E.
\]
This implies
\begin{equation}
\label{eq:q=5:MMM}
\tilde\MMM_4^2\cdot\tilde M_1=0,\qquad
\tilde\MMM_4\cdot\tilde M_1^2=-1.
\end{equation}
Note that $3\MMM_4$ is Cartier (see Lemma~\ref{lemma:K-index}).
Hence so is
\[
3\tilde\MMM_4+2E\sim f^*(3\MMM_4).
\]
Since $f$ is the Kawamata blowup of the point $P_3$ of type $\frac13 (1,1,2)$, 
the variety
$\tilde X$ has on $E$ a unique singularity that is of type $\frac12(1,1,1)$.
Hence $2E$ and $3\tilde\MMM_4$ are Cartier.
Since the indices of points of $\tilde X$ are coprime to $3$,
$\tilde\MMM_4$ is also Cartier.
Furthermore,
\[\textstyle
-K_{\tilde X}\qq f^*(-K_X)-\frac 13 E\qq 5f^*A_X-\frac 13E=\frac 
54\tilde\MMM_4+\frac 12 E,
\]
hence $-K_{\tilde X}$ is ample.
Since $\uprho(\tilde X)=2$, there exists an extremal Mori contraction $\bar 
f:\tilde X\to\hat X$ other than $f$.
Then~\eqref{eq:q=5:MMM} implies that $\tilde\MMM_4$ is not ample and is trivial 
on the fibers of $\bar f$.
By the base point free theorem we can write 
\begin{equation}
\label{eq:q=5:M4}
\tilde\MMM_4=\bar f^*\hat\MMM_4, 
\end{equation} 
where $\hat\MMM_4$
is a linear system of Cartier divisors on $\hat X$.
Moreover, the contraction $\bar f$ is birational and contracts $\tilde M_1$ to 
a curve
$\hat\Upsilon\subset\hat X$ such that 
$\hat\MMM_4\cdot\hat\Upsilon=-\tilde\MMM_4\cdot\tilde M_1^2=1$.
In particular, $\hat X$ is a $\QQ$-Fano threefold. By 
Lemma~\ref{lemma:sl:torsion}
the group $\Cl(\hat X)$ is torsion free and $\bar E=\bar f(E)$ is a generator 
of $\Cl(\hat X)$.
Since $-K_{\tilde X}\qq 5\tilde M_1+3E$, we have $\qW(\hat X)=3$ and 
\begin{equation}
\label{eq:q=5:M4E}
\hat\MMM_4\qq 2\hat E. 
\end{equation}
Since $\hat\MMM_4$ is Cartier, $\hat X$ is isomorphic to a hypersurface of 
degree $6$ in $\PP(1^2,2^2,3)$
by \cite[Theorem~1.5]{CampanaFlenner} (see also 
\cite[Theorem~2.2.5]{P:QFano-rat1}).
Then the linear system $\hat\MMM_4=|2A_{\hat X}|$ is base point free and defines
a finite degree $2$ morphism $\psi:\hat X\to Q'\subset\PP^4$, where 
$Q'\simeq\PP(1,1,2,2)$
is a quadric of corank $2$. Since $\hat\MMM_4\cdot\hat\Upsilon=1$,\ 
$\psi(\hat\Upsilon)$ is a line on $Q'$.
Therefore, $\hat\Upsilon$ is a smooth rational curve.
\end{proof}

\begin{sremark}
The restriction $\bar f_E: E\to\hat E$ coincides with the normalization.
Moreover, $\hat\Upsilon\subset\hat E$ is the nonnormal locus
and $\bar f_E^{-1}(\hat\Upsilon)= E\cap\tilde M_1$.
\end{sremark}

\begin{scorollary}
\label{cor:q=5:Bs3}
$P_3\notin\Bs|3A_X|$.
\end{scorollary}

\begin{proof}
Since $\dim\MMM_3=2$ and $\dim|A_{\hat X}|=1$, we have $s_3\ge 2$.
Then from~\eqref{eq:main1} for $k=3$ we obtain $s_3 =2$ and $\beta_3 =0$.
\end{proof}

\begin{scorollary}
\label{cor:q=5:P_2-P_3-P_4}
For general members $M_2\in \MMM_2$, $M_3\in \MMM_3$, and $M_4\in \MMM_4$ we 
have
\[
M_1\cap M_2\cap M_3 =\{P_4\}\quad \text{and}\quad M_1\cap M_2\cap M_3 \cap M_4 
=\varnothing.
\]
\end{scorollary}
\begin{proof}
By Corollary~\ref{cor:q=5:Bs3} we have $M_1\cap M_2\cap M_3\not\ni P_3$.
Assume that $M_1\cap M_2\cap M_3$ contains a curve, say $C$. Then
$M_4\cdot C\in \ZZ $ because $M_4$ is Cartier along $C$. On the other hand,
$M_4\cdot C\le M_1\cdot M_2\cdot M_4=2/3$, a contradiction.
Therefore, the set $M_1\cap M_2\cap M_3$ is zero-dimensional. Since $M_1\cdot 
M_2\cdot M_3=1/2$ and $M_1\cap M_2\cap M_3\ni P_4$,
we have $M_1\cap M_2\cap M_3 =\{P_4\}$.
\end{proof}

\begin{scorollary}
\label{cor:q5:s-beta}
In the above notation the following assertions hold.
\begin{enumerate}
\item
\label{cor:q5:s-beta:b}
$\hat\MMM_2=|A_{\hat X}|$ and $\hat\MMM_4=|2A_{\hat X}|$;

\item
\label{cor:q5:s-beta:bb}
$\hat\MMM_3\subset |2A_{\hat X}|$ and 
$\hat\MMM_5\subset |3A_{\hat X}|$ are subsystems of codimension two
consisting all the members passing through $\hat\Upsilon$;

\item
\label{cor:q5:s-beta:bbb}
$\Bs\hat\MMM_4=\varnothing$ and $\Bs\hat\MMM_2$ is an 
irreducible smooth rational curve that is different from $\hat\Upsilon$;

\item
\label{cor:q5:s-beta:a}
$\hat E\in |A_{\hat X}|$ and $\hat E$ is singular along $\hat\Upsilon$.

\end{enumerate}
\end{scorollary}
\begin{proof}
It follows from the proof of Proposition~\xref{prop:q=5o:main} that $e=1$ and 
$s_4=2$.
From~\eqref{eq:main1} we obtain $s_2\le 1$. Since $\MMM_2$ is movable, $s_2 >0$,
hence $s_1=1$ and $\hat\MMM_2\subset |A_{\hat X}|$.
Since $\dim|2A_X|=\dim|A_{\hat X}|=1$, we obtain $\hat\MMM_2=|A_{\hat X}|$. 
The equality $\hat\MMM_4=|2A_{\hat X}|$ follows from~\eqref {eq:q=5:M4} 
and~\eqref{eq:q=5:M4E}.
This proves~\ref{cor:q5:s-beta:b}.
For~\ref{cor:q5:s-beta:bb} we note that $ s_3\le 2$ and $ s_5\le 3$ 
by~\eqref{eq:main1}.
Since
$\dim|3A_X|=\dim|2A_{\hat X}|-2$ and 
$\dim|5A_X|=\dim|3A_{\hat X}|-2$, one can see that $\hat\MMM_3\subset |2A_{\hat 
X}|$ and 
$\hat\MMM_5\subset |3A_{\hat X}|$ are subsystems of codimension two.
Further, the relations~\eqref{eq:-b-gamma-delta-1} in our case have the form
\begin{equation*}
\begin{array}{lll}
b &=& 3\delta-5,
\\[2pt]
\gamma_k &=& s_k\delta-k,\qquad k=2,3,4,5.
\end{array}
\end{equation*}
Since $\bar f$ is generically a blowup of a curve, $b=1$.
Then $\delta=2$, $\gamma_2=\gamma_4=0$, and $\gamma_3=\gamma_5=1$.
The last equality means that
$\hat\Upsilon\subset\Bs\hat\MMM_3$ and $\hat\Upsilon\subset\Bs\hat\MMM_5$.
This proves~\ref{cor:q5:s-beta:bb}.
Since $\gamma_2=0$, we have $\hat\Upsilon\not\subset\Bs\hat\MMM_2$.
Denote $\hat\Gamma:=\Bs |A_{\hat X}|$.
It follows from the explicit equations (see \cite[Proposition~5.2]{P:P11223}) 
that
$\hat\Gamma$ can be given in $\PP(2,2,3)\simeq\PP(1,1,3)$ by one of the 
following equations
\[
x_3^2+x_2y_2(x_2+y_2)=0,\quad x_3^2+x_2^2y_2=0,\quad x_3^2+x_2^3=0.
\]
Then it is easy to see that $\hat\Gamma$ is smooth.
This proves~\ref{cor:q5:s-beta:bbb}.
The assertion~\ref{cor:q5:s-beta:a} follows from equalities $e=1$ and 
$\delta=2$.
\end{proof}

\begin{scorollary}
The collection of non-Gorenstein points of $\hat X$ is either
one index-$2$ point $\hat P_4$ with $\aw(\hat X,\hat P_4)=3$ or
two index-$2$ points $\hat P_4$ and $\hat P_2$ with $\aw(\hat X,\hat P_4)=2$ 
and $\aw(\hat X,\hat P_2)=1$.
In particular, $\hat X$ is not quasi-smooth.
The points $\hat P_2^{(i)}=\bar f(f^{-1}(P_2^{(i)}))$
are distinct and they are Gorenstein \textup(terminal\textup) singularities of 
$\bar X$.
\end{scorollary}

\begin{proof}
Since $\B(X)=(2^2,3,4)$ and $f$ is a Kawamata blowup of 
$P_3$, we have $\B(\tilde X)=(2^3,4)$. 
Thus the non-Gorenstein points of $\tilde X$ are as follows:
the index-$4$ point $\tilde P_4=f^{-1}(P_4)$ that is of the same type as 
$P_4\in X$,
either one or two index-$2$ points $\tilde P_2^{(i)}=f^{-1}(P_2^{(i)})$, 
$i=1,\dots, m$
that are also of the same type as $\tilde P_2^{(i)}\in X$, and the ``new'' 
index-$2$ point $\tilde P_2\in E$.
We distinguish two possibilities:
\begin{enumerate}
\renewcommand{\theenumi}{\rm\alph{enumi})}
\renewcommand{\labelenumi}{\rm\alph{enumi})}
\item
\label{case:q=5:P_2:n}
$\tilde P_2\notin\tilde M_1$,
\item 
\label{case:q=5:P_2:i}
$\tilde P_2\in\tilde M_1$.
\end{enumerate}
Let $\tilde M_3\in\tilde\MMM_3$ be a general member and let $\Lambda:=\tilde 
M_3\cap\tilde{M}_1$.
Note that the points $P_4$, $P_2^{(1)}$,\dots, $P_2^{(m)}$ lie on the curve 
$f(\Lambda)$
and $f(\Lambda)\not\ni P_3$ because $\beta_3=0$ by~\eqref{eq:main1}. Hence we 
have
$\tilde{P}_4$, $\tilde {P}_2^{(1)},\dots,\tilde {P}_2^{(m)}\in\Lambda$, 
\[
2A_{\hat{X}}\cdot\bar{f}(\Lambda)=\tilde\MMM_4\cdot\Lambda=\tilde\MMM_4\cdot 
M_3\cdot\tilde{M}_1=(4A_X)\cdot (3 A_X)\cdot A_X=1.
\]
Since $2A_{\hat{X}}$ is an ample Cartier divisor, $\bar{f}(\Lambda)$ is 
irreducible.
Since $f(\Lambda)$ does not pass through $P_3$, the divisor
$4A_X$ is Cartier along $f(\Lambda)$ and $f(\Lambda)$ is irreducible because 
$f^*(4A_X)\cdot\Lambda=1$.
Therefore, $\Lambda$ is irreducible as well and 
$\bar{f}$ induces an isomorphism between $\Lambda$ and 
$\hat\Upsilon$.
This implies, that the points 
$P_4$, $P_2^{(1)},\dots,P_2^{(m)}$ lie on different fibers of 
$\bar{f}_{\tilde{M}_1}:\tilde{M}_1\to\hat\Upsilon$.
In particular, the points 
$\bar{f}(\tilde{P}_4)$, $\bar{f}(\tilde {P}_2^{(1)}),\dots,\bar{f}(\tilde 
{P}_2^{(m)})$ are distinct.
Since the curve $\hat\Upsilon$ is smooth, for any singular point $\tilde 
P\in\tilde X$ the variety $\hat X$ 
must be singular at $\bar f(\tilde P)$. 
In particular, the points 
$\bar{f}(\tilde{P}_4)$, $\bar{f}(\tilde {P}_2^{(1)}),\dots,\bar{f}(\tilde 
{P}_2^{(m)})$ are
singular.
By \cite[Theorem~4.7]{KM:92} the points
$\bar{f}(\tilde {P}_2^{(1)}),\dots,\bar{f}(\tilde {P}_2^{(m)})$ are Gorenstein.
On the other hand, the point $\hat P_4$ must be of index~$2$ because $\B(\hat 
X)=(2^3)$.
This point is not a cyclic quotient by \cite{Kawamata:Div-contr}.
Moreover, $\aw(\hat X,\hat P_4)=2$ in the case~\ref{case:q=5:P_2:n} and
$\aw(\hat X,\hat P_4)=3$ in the case~\ref{case:q=5:P_2:i}.
\end{proof}

\begin{proposition}
\label{prop:q=5:link-cb}
In the notation of Proposition~\xref{prop:q=5:link2} the link has the following 
form
\begin{equation}
\label{eq:sl:q=5-2}
\vcenter{
\xymatrix@R=1em{
&\tilde X\ar@/_0.4em/[dl]_{f}\ar@/^0.4em/[dr]^{\bar f}&
\\
X&&\PP(1,2,3)
}}
\end{equation}
where $\bar f$ is a $\QQ$-conic bundle with discriminant curve $\Delta_{\bar 
f}\in |-2K_{\PP(1,2,3)}|$.
\end{proposition}

\begin{proof}
First, we claim that 
the linear system $\tilde\MMM_3$ is nef.
Indeed, assume that $\tilde\MMM_3\cdot\tilde C<0$ for some irreducible curve 
$\tilde C$.
Then $\tilde C\subset\Bs\tilde\MMM_3$ and $\tilde C\cap E\neq\varnothing$, hence
$E\cap\Bs\tilde\MMM_3\neq\varnothing$ and $P_3\in\Bs\MMM_3$. 
This contradicts Corollary~\ref{cor:q=5:Bs3}. Thus $\tilde\MMM_3$ is nef.
Now it follows from~\eqref{eq:q=5:m} that $\beta_3=3/4$, i.e. $\tilde\MMM_3\qq 
f^*\MMM_3-\frac{3}{4}E$. Then
the relation
\[\textstyle
-K_{\tilde X}\qq f^*(5A_X)-\frac14 E\qq\frac 53\tilde\MMM_3+E
\]
implies that $-K_{\tilde X}$ is ample because $\tilde\MMM_3$ is nef. Since 
$\bar\MMM_3=\bar f^* |3A_{\hat X}|$,
the linear system is not big and neither is $\tilde\MMM_3$. Hence 
$\tilde\MMM_3$ is a supporting linear system for
an extremal nonbirational Mori contraction $\tilde X\to\check X$.
This implies that in the diagram~\eqref{diagram-main} the map $\chi$ is an 
isomorphism and $\check X=\hat X$.
Further, 
\[\textstyle
0=(\tilde\MMM_3)^3=(3A_X)^3-(\frac{3}{4}E)^3=\frac 94-\frac{27}{64}E^3,
\]
which implies $E^3=16/3$. Then by~\eqref{eq:cb:a} and~\eqref{eq:cb:b} we have
\begin{eqnarray*}
A_{\hat X}^2&=&
\textstyle
-\frac12 K_{\tilde X}\cdot\tilde M_1^2=\frac12(5f^*A_X-\frac14 E)\cdot 
(f^*A_X-\frac14 E)^2=5A_X^3-\frac 1{64}E^3=\frac 16,
\\
A_{\hat X}\cdot\Delta_{\bar f}&=&
\textstyle
-K_{\tilde X}^2\cdot\tilde M_1-4K_{\hat X}\cdot A_{\hat X}=
-(5f^*A_X-\frac14 E)^2\cdot (f^*A_X-\frac14 E)+24 A_{\hat X}^2=2.
\end{eqnarray*}
Hence $\hat X\simeq\PP(1,2,3)$ (see Corollary~\ref{cor:base-surface})
and $\Delta_{\bar f}\sim 12A_{\hat X}\sim -2 K_{\hat X}$.
\end{proof}

\begin{scorollary}
\label{cor:q=5:non-rat}
Assume that every non-Gorenstein singularity of $X$ is a cyclic quotient 
and every Gorenstein singularity is either node or cusp. 
Furthermore, assume that the number of Gorenstein singularities is at most two.
Then $X$ is not rational.
\end{scorollary}

\begin{proof}
We have an isomorphism $X\setminus \{P_4\}\simeq \tilde X\setminus E$. 
Hence every non-Gorenstein singularity of $\tilde X$ is a cyclic quotient 
and $\B(\tilde X)=(2^2,3^2)$.
By Proposition~\ref{prop:cb-index2} the germ of $\bar f$ over the 
type~\type{A_1} point $o_1\in \PP(1,2,3)$
has type~\typeci{T}{2} and over the type~\type{A_2} point $o_2\in \PP(1,2,3)$ 
it
has type~\typeci{T}{3}. In particular, Gorenstein singularities of $\tilde X$ 
are not contained in the fibers over 
$o_1$ nor $o_2$.
By Theorem~\ref{thm:cb} the variety $\tilde X$ is not rational.
\end{proof}

\begin{theorem}
\label{thm:q5}
Let $X$ be a $\QQ$-Fano threefold with $\qQ(X)=5$.
Assume that $X$ is not rational and at least one of the following holds:
\begin{enumerate}
\item 
\label{thm:q5-1}
$\B(X)=(2^2, 3, 4)$, 
\item 
\label{thm:q5-2}
$A_X^3=1/12$ and $\g(X)\ge 5$,
\item 
\label{thm:q5-3}
$\p_2(X)\ge 2$.
\end{enumerate}
Then $X$ is isomorphic to a hypersurface $X_{10}\subset\PP(1,2,3,4,5)$.
In particular, the group $\Cl(X)$ is torsion free.
\end{theorem}

\begin{proof}
By \cite[Proposition~7.4]{P:QFano-rat1} for a nonrational $\QQ$-Fano threefold
$X$ with $\qQ(X)=5$ the condition~\ref{thm:q5-3} implies both
~\ref{thm:q5-1} and~\ref{thm:q5-2}. Conversely, the computer search shows that 
either~\ref{thm:q5-1} or~\ref{thm:q5-2} implies~\ref{thm:q5-3}.
Therefore, in the assumptions of Theorem~\ref{thm:q5} the 
conditions~\ref{thm:q5-1},~\ref{thm:q5-2}, 
~\ref{thm:q5-3} are equivalent.
Thus for the proof of Theorem~\ref{thm:q5} we may assume that $X$ is a 
$\QQ$-Fano threefold 
satisfying the conditions~\xref{setup:q=5}.

For $m=1,\dots,5$, let $\varsigma_m$ be a general element of $ H^0 (X, mA_X )$.
By Corollary \ref{cor:q=5:P_2-P_3-P_4} the map
\[
\Psi: X \dashrightarrow \PP(1,2,3,4),\qquad P \longmapsto 
(\varsigma_1(P),\varsigma_2(P),\varsigma_3(P),\varsigma_4(P)).
\]
is a morphism. For short, let $\PP:=\PP(1,2,3,4)$ and let $A_{\PP}$ be the 
positive generator of $\Cl(\PP)$.
By the construction, $A_X=\Psi^* A_{\PP}$.
Since $A_X^3=1/12$
and $A_{\PP}^3=1/24$, the morphism $\Psi$ is finite of degree $2$.
By the Hurwitz formula we have
\[
\textstyle
\Psi^* (5A_{\PP})=5A_X = -K_X=\Psi^* \left(-K_{\PP} +\frac 12 R\right)
=\Psi^* \left(10 A_{\PP} -\frac 12 R\right),
\]
where $R$ is the branch divisor. This gives us $R\sim 10 A_{\PP}$. Thus
$X\to \PP(1,2,3,4)$ is a double cover branched over a divisor $R\sim 10 
A_{\PP}$.
Then $X$ must be a hypersurface of degree~$10$ in $\PP(1,2,3,4,5)$.
\end{proof}

\begin{proposition}
\label{prop:q5-eq-1}
Let $X$ be 
a hypersurface of degree $10$ in $\PP(1,2,3,4,5)$.
Assume that the singularities of $X$ 
are terminal. Then in some coordinate system $x_1,x_2,x_3,x_4,x_5$ the equation 
of $X$ can be written in one of following forms:
\begin{align}
\label{eq:thm:q5-1:a}
&&x_5^2+x_4^2x_2+x_4\phi_6(x_1,x_3)+\phi_{10}(x_1,x_2,x_3) =0,
\\
\label{eq:thm:q5-1:b}
&&x_5^2+x_4x_3^2+\lambda x_4^2x_1^2+x_4\phi_6(x_1,x_2)+\phi_{10}(x_1,x_2,x_3) 
=0,
\end{align}
where $\lambda$ is a constant, $x_i$ is a variable of degree $i$
and, in the case~\eqref{eq:thm:q5-1:a}, either $\phi_6\ni x_3^2$ or 
$\phi_{10}\ni x_3^3x_1$.
The index-$4$ point is a cyclic quotient in the case~\eqref{eq:thm:q5-1:a}
and has type~\typeA{cAx}{4} in the case~\eqref{eq:thm:q5-1:b}.
\end{proposition}

\begin{proof}
Let $\phi=0$ be an equation of $X$. If $\phi$ does not contain $x_5^2$, then 
the point $P_5$ lies on $X$ and
it is the quotient of a hypersurface singularity by $\mumu_5(1,2,3,4)$. Such a 
point cannot be terminal.
Thus $\phi\ni x_5^2$. Completing the square we may assume that $\phi$ does not 
contain other terms that depend on $x_5$.
The point $P_3$ lies on $X$ and must be a cyclic quotient singularity, hence 
$\phi$ contains either $x_1x_3^3$ or $x_4x_3^2$. 
Further, $P_4\in X$ and if it is a cyclic quotient, then 
$\phi\ni x_4^2x_2$ and we are done with~\eqref{eq:thm:q5-1:a}. 
Otherwise $P_4\in X$ is the quotient of a hypersurface singularity by 
$\mumu_4(1,2,3,1)$.
Such a point must be of type~\typeA{cAx}{4} and so $\phi\ni x_4x_3^2$. 
Then the equation can be reduced to~\eqref{eq:thm:q5-1:b}.
\end{proof}

\begin{scorollary}
\label{cor:q5:ind2sing}
 Let $X$ be a $\QQ$-Fano threefold with $\qQ(X)=5$ and $\B(X)=(2^2, 3, 4)$
\textup(see Theorem~\xref{thm:q5}\textup). If $X$ has an index-$2$ point $P\in X$ that not a cyclic 
quotient singularity, then $P\in X$ is of type \typeA{cA}{2} or \typeA{cAx}{2}.
In this case the index-$4$ point is a cyclic 
quotient singularity.
\end{scorollary}

\begin{scorollary}
\label{cor:q5:nonrat}
Let $X$ be a $\QQ$-Fano threefold with $\qQ(X)=5$ and $\B(X)=(2^2, 3, 4)$
\textup(see Theorem~\xref{thm:q5}\textup).
If the index-$4$ point is not a cyclic 
quotient singularity, then $X$ is rational.
\end{scorollary}

\begin{proof}
Assume that $X$ is not rational.
By Theorem~\ref{thm:q5}\ $X$ is a hypersurface $X_{10}\subset\PP(1,2,3,4,5)$ 
and we may assume that its equation has the form~\eqref{eq:thm:q5-1:b} (see 
Proposition~\xref{prop:q5-eq-1}). 
If $\lambda=0$, then the projection $X \dashrightarrow \PP(1,2,3,5)$ is a 
birational map, which proves the rationality 
of $X$. Thus we may assume that $\lambda=-1$ and then in the affine chart 
$x_1\neq 0$ this equation 
can be written as follows:
\[
(x_5-x_4)(x_5+x_4)+x_4x_3^2+x_4\phi_6(1,x_2)+\phi_{10}(1,x_2,x_3) =0.
\]
Applying the linear coordinate change $x_5'=x_5-x_4$, $x_4'=x_5+x_4$ one can 
see that $X$ is rational.
\end{proof}

\section{$\QQ$-Fano threefolds of index~$6$ and $7$: Rationality}
\label{sect:67}
In this section we generalize Theorem~\ref{thm0} in the cases $\qQ(X)=7$ and $6$.
\begin{proposition}
\label{prop:q7a}
Let $X$ be a $\QQ$-Fano threefold with $\qQ(X)=7$.
If $\p_1(X)>0$, then $X$ is rational.
\end{proposition}

\begin{proof}
Let $X$ be a $\QQ$-Fano threefold with $\qQ(X)=7$ and $\p_1(X)>0$. Assume that 
$X$ is not rational. 
By Proposition~\ref{prop:t} the 
group $\Cl(X)$ is torsion free,
hence $|A_X|\neq\varnothing$.
Recall that by 
Theorem~\ref{thm0} we have $\p_3(X)\le 1$, hence by
Proposition~\xref{prop:6-7}
there are only two numerical possibilities:~\ref{Case:q=7:B=2-6-10} and
~\ref{Case:q=7:B=2-313} in Table~\ref{tab:1}.
The former case is impossible by Lemma~\ref{prop:q=7:tor}.
Consider the latter one.

\subsection*{Case~\ref{Case:q=7:B=2-313}}
Then $\B(X)=(2,3,13)$. By Proposition~\xref{prop:6-7} we have 
$\dim|A_X|=\dim|4A_X|=0$ and $\dim|6A_X|=1$.
So, the linear system $|6A_X|$ is movable and has no fixed components.
Apply the construction~\eqref{diagram-main} with $\MMM=|6A_X|$.
If $P$ is the point of index~$13$, then $\MMM 
\overset{P}{\sim} 12(-K_X)$ and so $\beta\ge 12\alpha$
by~\eqref{eq:ct:beta-alpha}.
Since $6M_1\in\MMM$, we have $6\beta_1\ge\beta$.
Hence $\beta_1\ge\frac 16\beta\ge 2\alpha$. 
The relation~\eqref{eq:main1} for $k=1$ and~\eqref{eq:main} have the form
\begin{align}
\label{eq:main:q=7a:1}
\hat q&=7s_1+(7\beta_1-\alpha)e\ge 7s_1+ 13\alpha e\ge 7s_1+ e,
\\
\label{eq:main:q=7a:6}
6\hat{q}&=7 s+(7\beta-6\alpha) e\ge 7 s+78\alpha e\ge 7 s+6 e.
\end{align}
Since $\hat q\le 7$, from~\eqref{eq:main:q=7a:1} we obtain $s_1=0$.

Assume that $\hat q=1$. Then $s=s_1=0$ by~\eqref{eq:main:q=7a:6}. Hence, 
$\bar f$ is a fibration.
If $\hat X$ is a surface, then
$\bar\MMM=\bar f^*\NNN$ and $\bar M_1=\bar f^* N_1$ for some complete linear 
system $\NNN$
and Weil divisor $N_1$
on $\hat X$ such that $\dim\NNN=1$ and $\dim|N_1|=0$.
Therefore, $\hat X$ is either
$\PP(1,2,3)$ or $S_{\mathrm{DP}_5}$ (see Corollary~\ref{cor:base-surface}).
But then $\NNN\sim 2N_1$ and so $\bar\MMM\sim 2\bar M_1$. Pushing this relation 
down to $X$
we obtain $6A_X\sim\MMM\sim 2 M_1\sim A_X$, a contradiction.
Therefore, $\hat X\simeq\PP^1$ and $\bar f$ is a del Pezzo fibration.
Then $\bar M_1$ is a multiple fiber of multiplicity $6$. In this case
by \cite{MP:DP-e}
the general fiber $\bar X_\eta$ of $\bar f$ is a smooth del Pezzo surface of 
degree $6$ over the field
$\CC(\hat X)$. Since $\CC(\hat X)$ is a $\mathrm{c}_1$-field, the surface $\bar 
X_\eta$ has a $\CC(\hat X)$-point.
Hence $\bar X_\eta$ is rational over $\CC(\hat X)$ (see e.~g. 
\cite[Theorem~29.4]{Manin:book:74}) and so is $\bar X$ over $\CC$, a 
contradiction.

Thus $\hat q>1$. Then $\bar f$ is birational and $s>0$ 
\cite[Lemma~5.1]{P:QFano-rat1}.
Since $s_1=0$, the group $\Cl(\hat X)$ is torsion free and $e=1$ (see 
Lemma~\ref{lemma:sl:torsion}).
If $\alpha\ge 1/3$, then
$\hat q\ge 6$ and $s\le 2$. This contradicts by Theorem~\ref{thm0}.
Thus $f(E)=P$ and $\alpha=1/13$.
Then~\eqref{eq:main:q=7a:1} gives us $13\beta_1=(13\hat{q}+1)/7$.
Since $13\beta_1$ is an integer and $\hat q\le 7$, the only possibility is 
$\hat q=1$, a contradiction.
\end{proof}

\begin{proposition}
\label{prop:q6}
Let $X$ be a $\QQ$-Fano threefold with $\qQ(X)=6$.
Then $X$ is rational.
\end{proposition}

\begin{proof}
Let $X$ be a $\QQ$-Fano threefold with $\qQ(X)=6$.
Assume that $X$ is not rational. Then $\p_3(X)\le 1$ by Theorem~\ref{thm0} and 
we have one of the cases~\ref{Case:q=6:B=5-17}, \ref{Case:q=6:B=7-11}, 
\ref{Case:q=6:B=5-11} of Table~\ref{tab:1}. 
In all these cases the group $\Cl(X)$ is torsion free and $|5A_X|$ is a pencil 
without fixed components.
Apply the construction~\eqref{diagram-main} with $\MMM=|5A_X|$.
If $P$ is the point of index~$11$ (resp. index~$17$), then $\MMM 
\overset{P}{\sim} 10(-K_X)$ (resp. $\MMM 
\overset{P}{\sim} 15(-K_X)$) and so $\beta_5\ge 10\alpha$
by~\eqref{eq:ct:beta-alpha}.
The relation~\eqref{eq:main} has the form
\begin{equation}
\label{eq:q=6}
5\hat{q}=6 s_5+(6\beta_5-5\alpha) e\ge 6 s_5+55\alpha e.
\end{equation}

First, consider the cases where $s_5=0$. Then $\bar f$ is a fibration, $\hat 
q=1$, 
$e=1$, and $\alpha\le 1/11$.
In this case $\bar\MMM=\bar f^*\NNN$ for some complete linear system with 
$\dim\NNN=1$.
Assume that $|A_X|\neq\varnothing$.
Then $\dim|A_X|=\dim|2A_X|=0$.
From~\eqref{eq:main} we obtain 
$s_1=0$, hence $\bar M_1$ is $\bar f$-vertical and $5\bar M_1\sim\bar\MMM$.
If $\hat X$ is a surface, then $\bar M_1=\bar f^* N_1$, where $N_1$ is an 
effective divisor such that $\dim|N_1|=0$.
In this case $\NNN\sim 5N_1$. This contradicts 
Corollary~\ref{cor:base-surface}. Thus $\hat X\simeq\PP^1$ and 
$\NNN=|\OOO_{\PP^1}(1)|$.
Since $5\bar M_1\sim\bar\MMM$,\ $\bar M_1$ is
a fiber of multiplicity $5$. 
By \cite{MP:DP-e} the general fiber $\bar X_\eta$ of $\bar f$ is a smooth del 
Pezzo surface of degree $5$ over the field
$\CC(\hat X)$. Then $\bar X_\eta$ is rational over $\CC(\hat X)$ (see e.~g. 
\cite{Shepherd-Barron1992}) and so is $\bar X$ over $\CC$, a 
contradiction.
Therefore, $|A_X|=\varnothing$. Then $\dim|2A_X|=\dim|3A_X|=0$ and so 
the linear system $|2A_X|$ (resp. $|3A_X|$) consists of a unique irreducible
divisor $M_2$ (resp. $M_3$). From~\eqref{eq:main} we obtain 
$s_2=s_3=0$, hence
the divisors $\bar M_2$ and $\bar M_3$ are 
vertical with respect to
$\bar f$, that is, they do not dominate $\bar f$.
If $\hat X$ is a surface, then
$\bar M_2=\bar f^* N_2$ and $\bar M_3=\bar f^*N_3$ for some effective 
Weil divisors $N_2,\, N_3$
such that $\dim|N_1|=\dim|N_2|=0$. 
Since $\Cl(\hat X)\simeq\ZZ$, we have $N_2\sim N_3$, a contradiction.
Hence $\hat X\simeq\PP^1$ and $\NNN=|\OOO_{\PP^1}(1)|$.
Then $\bar M_2$ and $\bar M_3$ are contained in fibers.
Since the divisors $\bar M_2$ and $\bar M_3$ are 
reduced and irreducible, they must be multiple fibers: $m_2\bar M_2\sim 
m_3\bar M_3\sim\bar\MMM$.
Pushing down this relations to $X$ we obtain $m_2 M_2\sim m_3 M_3\sim\MMM\sim 5 
A_X$,
a contradiction. 

Thus $s_5>0$, $\hat q >1$, and $\bar f$ is birational (see 
Remark~\ref{rem:case-nonbir}).
If $\alpha\ge 1$, then $\hat q>7$ by~\eqref{eq:q=6}. The contradiction 
shows that $\alpha=1/r$, where $r\in\{5,\, 7,\, 11,\, 17\}$ because the 
singularities of $X$ are cyclic quotients. 
Then from~\eqref{eq:q=6} one obtains two possibilities:
\begin{itemize}
\item 
$\hat q=7$, $e=1$;
\item 
$\hat q=5$, $e=1$, $r=7$, and $s_5 \le 2$.
\end{itemize}
In the former case we have $\p_1(\hat X)>0$. This contradicts 
Proposition~\ref{prop:q7a}. 
In the latter case,
we are in the situation of~\ref{Case:q=6:B=7-11} in 
Table~\ref{tab:1} and $\p_2(\hat X)\ge 2$. Hence
$|A_X|=\varnothing$ and
$\Clt{X}\neq 0$ by Lemma~\ref{lemma:sl:torsion}. 
This contradicts 
Theorem~\ref{thm:q5}.
\end{proof}

\section{Nonrational $\QQ$-Fano threefolds of index~$7$}
\label{sect:q=7a}
In this section we study $\QQ$-Fano threefolds of index~$7$ that can potentially be nonrational.
According to our results in the previous section and  Proposition~\xref{prop:6-7}
there are only three ``numerical candidates''.
\begin{theorem}
\label{thm:q7b}
Let $X$ be a $\QQ$-Fano threefold with $\qQ(X)=7$. 
Assume that $X$ is not rational. Then $X$ is birationally equivalent to
a $\QQ$-Fano hypersurface of degree $10$ in $\PP(1,2,3,4,5)$.
\end{theorem}
The rest of this section is devoted to the proof of this theorem.
Let $X$ be a $\QQ$-Fano threefold with $\qQ(X)=7$. 
Assume that $X$ is not rational. Then $\p_3(X)\le 1$ and $\p_1(X)=0$ 
by Theorem~\ref{thm0} and Proposition~\xref{prop:q7a}. By
Proposition~\xref{prop:6-7} we have only three numerical possibilities:
~\ref{Case:q=7:B=2-2-2-3-4-5}, \ref{Case:q=7:B=2-2-2-5-8}, 
\ref{Case:q=7:B=3-8-9}
in Table~\ref{tab:1}. 
Thus, Theorem~\ref{thm:q7b} is a consequence of Propositions~\ref{prop:q=7},
\ref{prop:q=7:B=3-8-9}, and~\ref{prop:q=7:B=2-2-2-5-8} below.

\subsection{Case~\xref{Case:q=7:B=2-2-2-3-4-5}.}
Then $\B(X)=(2^3, 3, 4, 5)$.
Denote by $P_3$, $P_4$, and $P_5$ points of index~$3$, $4$, and $5$, 
respectively.
\begin{sproposition}
\label{prop:q=7}
Let $X$ be a $\QQ$-Fano threefold of type~\xref{Case:q=7:B=2-2-2-3-4-5} in Table~\xref{tab:1}.
Assume that $X$ is not rational. 
Then there exists a Sarkisov link of the form~\eqref{diagram-main}, where 
$\hat{X}$ is a hypersurface of degree $10$ in $\PP(1,2,3,4,5)$, $f$ is the 
Kawamata blowup of 
the point of index~$5$, and $\bar{f}$ contracts
the proper transform of the unique element $M_3\in |3A_X|$ to a curve.
\end{sproposition}

More precise description of this link will be given in Proposition~\ref{prop:q=7m}.

\begin{proof}
In this case the linear system $|4A_X|$ is a pencil without fixed components.
Apply the construction~\eqref{diagram-main} with $\MMM=|4A_X|$.
If $P$ is the point of index~$5$, then $\MMM\overset{P}{\sim} 2(-K_X)$ and so 
$\beta\ge 2\alpha$
by~\eqref{eq:ct:beta-alpha}. The relation~\eqref{eq:main} has the form
\begin{equation}
\label{eq:q=7:mainn}
4\hat{q}=7 s+(7\beta-4\alpha) e\ge 7 s+10\alpha e\ge 7 s+2e.
\end{equation}

Assume that $\hat q=1$. 
Then $s=0$, $\alpha\le 1/3$, and $e\le 2$ because $\alpha\ge 1/5$.
Therefore, $f(E)$ is a point of index $r\in\{3,4,5\}$ and
$\alpha=1/r$, hence $\beta\in\frac 1r\ZZ$. From~\eqref{eq:q=7:mainn} we 
obtain 
$7\beta=4/e+4/r\ge 14/r$ and $2r\ge 5 e$.
If $e=1$, then $\beta=4(r+1)/(7 r)\notin\frac 1r\ZZ$, a contradiction.
Therefore, $e=2$ and $r=5$. Then from~\eqref{eq:main1} for $k=3$ we obtain
\begin{equation*} 
3=3\hat{q}=7 s_3+(7\beta_3-3\alpha) e=7 s_3+14\beta_3-6/5,\quad s_3+2 
\beta_3=3/5.
\end{equation*}
The only possibility is $s_3=0$ and then $\beta_3=3/10\notin\frac1r\ZZ$, a 
contradiction.

Thus $\hat q>1$ and $s>0$. By~\eqref{eq:q=7:mainn} we have $\hat q\ge 3$.
If $\hat q=3$, then $s=1$, $\p_1(\hat X)\ge 2$, and then the group $\Cl(X)$ 
is torsion free by Corollary~\ref{cor:q=3:tor}.
Hence $e>1$ by Lemma~\ref{lemma:sl:torsion}.
As above, $7\beta=4/r+5/e\ge 14/r$ and $r\ge 2e$.
Since $\beta\in\frac1r\ZZ$, we get a contradiction.

Therefore, $\hat q\ge 4$, $\p_1(\hat X)\le 1$, and $s\ge 2$ by 
Theorem~\ref{thm0}. 
If $e=1$, then $\Clt{\hat X}\neq 0$ by Lemma~\ref{lemma:sl:torsion} because $|A_X|=\varnothing$.
In this situation, $\hat q=4$ by Proposition~\xref{prop:t} and $\Clt{\hat X}\simeq \ZZ/2\ZZ$ by
Corollary~\ref{cor:q=3:tor}. 
Hence, $\qW(\hat X)\neq\qQ(\hat X)$ (see Lemma~\ref{lemma:qQ=qW}).
Further, \eqref{eq:q=7:mainn}
implies $s=2$ and $\alpha=1/5$. 
By \cite[Proposition~3.2]{P:QFano-rat1} there exists only one effective 
divisor $D$ such that $D\qq A_{\hat X}$.
Since $e=1$, $D=\hat E$. On the other hand, from~\eqref{eq:main1} we obtain $s_2,\, s_3\in\{0,\, 1\}$, so 
either $D=\hat M_2$ or $\hat M_3$, a contradiction.

Thus $e>1$  and $\hat q\ge 5$. 
If $\hat q\ge 6$, then $s\ge 4$ by 
Theorem~\ref{thm0}. This contradicts~\eqref{eq:q=7:mainn}.
It remains to consider the case $\hat q=5$. Then $s=2$, hence $\p_2(\hat 
X)\ge 2$.
By Theorem~\ref{thm:q5}
the only possible case is where $\hat X$ is isomorphic to a hypersurface 
$X_{10}\subset\PP(1,2,3,4,5)$.
From~\eqref {eq:main1} we obtain 
$e=3$, $s_2 =1 $, $s_3 =0$,   $s_6 =3$, that is,
$\hat E\sim 3A_{\hat X}$, $M_2\sim A_{\hat X}$, $\hat\MMM\subset |2A_{\hat 
X}|$,
$\hat\MMM_6\subset |3A_{\hat X}|$.
Comparing the dimensions of linear systems we see that
$\hat\MMM=|2A_{\hat X}|$ and $\hat\MMM_6=|3A_{\hat X}|$.
The relations~\eqref{eq:-b-gamma-delta-1} have the form
\begin{align*}
3b&=5\delta -7,
\\
3\gamma_2=\gamma_6&=\delta-2,
\\
3\gamma_4&=2\delta-4.
\end{align*}
Therefore, $\delta\ge 2$ and $b\ge 1$. If $\delta> 2$, then 
$\gamma_2,\gamma_4,\gamma_6>0$.
This means that $\bar f(\bar F)$ is contained in the base loci of linear 
systems 
$|2A_{\hat X}|$ and $|3A_{\hat X}|$.
On the other hand, it can be seen from the equation~\eqref{eq:thm:q5-1:a}
that $\Bs |2A_{\hat X}|\cap\Bs |3A_{\hat X}|$ is a single point that has index~$4$.
But in this case $b=1/4$ by Lemma~\ref{lemma-discrepancies}, a contradiction.
Hence $\delta=2$, $b=1$, $\gamma_2=\gamma_4=\gamma_6=0$, and $\gamma_5=1$.
This shows that $\bar f(\bar F)$ is a curve that is  not contained in $\Bs |2A_{\hat X}|$ nor in $\Bs |3A_{\hat X}|$.
\end{proof}

\begin{scorollary}
\label{cor:q=7:beta}
In the notation of Proposition~\xref{prop:q=7} we have
$\beta_k<1$ for $k=2,\dots,6$.
\end{scorollary}

\begin{scorollary}
\label{cor:q=7:P_5}
$P_5\notin\Bs|5A_X|$.
\end{scorollary}
\begin{proof}
Follows from the equality $\beta_5=0$.
\end{proof}

\subsection{Case~\xref{Case:q=7:B=3-8-9}.}
Then $\B(X)=(3,8,9)$, $|A_X|=|2A_X|=\varnothing$,  $\dim|3A_X|=\dim|4A_X|=0$, and  $\dim|5A_X|=\dim|6A_X|=1$.
\begin{sproposition}
\label{prop:q=7:B=3-8-9}
Let $X$ be a $\QQ$-Fano threefold of type~\xref{Case:q=7:B=3-8-9} in Table~\xref{tab:1}. 
Assume that $X$ is not rational. Then there exists 
a Sarkisov link of the form~\eqref{diagram-main}, where $f$ is the Kawamata 
blowup of 
the point of index~$9$, the variety $\hat X$ is isomorphic to a hypersurface 
$\hat X_{10}\subset\PP(1,2,3,4,5)$, and
$\bar f$
contracts a divisor to a Gorenstein or smooth point.
\end{sproposition}

\begin{proof}
In this case the linear system $|6A_X|$ is a pencil without fixed components.
Apply the construction~\eqref{diagram-main} with $\MMM=|6A_X|$.
If $P$ is the point of index~$9$, then $\MMM\overset{P}{\sim} 6(-K_X)$ and so 
$\beta\ge 6\alpha$
by~\eqref{eq:ct:beta-alpha}. The relation~\eqref{eq:main} has the form
\begin{equation}
\label{eq:7:1/72}
6\hat{q}=7 s+(7\beta-6\alpha) e\ge 7 s+36\alpha e\ge 7 s+4e.
\end{equation}
If $\alpha\ge 1$, then $\hat q\ge 6$, $s\ge 2$, and $\hat q>7$, a 
contradiction.
Thus $f(E)$ is a cyclic quotient singularity of index $r\in\{3,8,9\}$ and 
$\alpha=1/r$.
The number $\lambda:=\beta r$ is integral, $\lambda\ge 6$ 
and~\eqref{eq:7:1/72} can be rewritten in the following form
\begin{equation*}
%\label{eq:7:1/72}
6\hat{q}r=7 sr+(7\lambda-6 ) e\ge 7 sr+36 e,
\qquad
6 (\hat{q}r+e)=7 (sr+\lambda e).
\end{equation*}
If $\hat q=1$, then $s=0$, $e=1$, and $6 (r+1)\equiv 0\mod 7$, a 
contradiction.
Hence, $\hat q>1$, $\bar f$ is birational and $s\ge 1$.
If $\hat q=2$, then again $s=e=1$, and $6 (2r+1)\equiv 0\mod 7$, hence $r=3$ 
and 
$\lambda=3<6$, a contradiction.
Hence, $\hat q\ge 3$. 
Assume that $\hat q=3$. If $e\ge 2$, then $s=1$, $e=2$.
In this case $r$ must be even, so $r=8$ and $\lambda$ is not an integer, a 
contradiction.
Hence, $e=1$. Then $s\le 2$.
Since $|A_X|=|2A_{X}|=\varnothing$, $|\Clt{\hat X}|\ge 3$ by 
Lemma~\ref{lemma:sl:torsion}.
Then by Corollary~\ref{cor:q=3:tor} the variety $\hat X$ is rational, a 
contradiction.
Thus $\hat q\ge 4$ and by Theorem~\ref{thm0} we may assume that $s\ge 2$

If $\hat q=4$. 
Then $s=2$, $e\le 2$, and $4r+e\equiv 0\mod 7$.
The only possibility is $r=3$, $e=2$, and $\lambda=3<6$, a contradiction.
If $\hat q=6$, then $\hat X$ is rational by Proposition~\xref{prop:q6}.
If $\hat q=7$, then $e$ must be divisible by $7$ and $s\ge 4$ 
by Theorem~\ref{thm0}. This contradicts~\eqref{eq:7:1/72}.

Finally, assume that 
$\hat q=5$. Then $s\le 3$.
If $s=3$, then $e\le 2$ and $e=1$ because $sr+\lambda e\equiv 0\mod 2$.
But then $\hat{q}r+e=5r+1\mod 7$, which is impossible for $r\in\{3,8,9\}$.
Thus $s=2$
and so $\p_2(\hat X)\ge 2$. 
By Theorem~\ref{thm:q5}
the variety $\hat X$ is isomorphic to a hypersurface 
$X_{10}\subset\PP(1,2,3,4,5)$.
Hence $\Cl(\hat X)=0$ and $e>2$ by Lemma~\ref{lemma:sl:torsion}.
From~\eqref{eq:7:1/72} and~\eqref{eq:main} with $k=3$, $4$, $5$
we obtain $r=9$, $\alpha=1/9$, $e=4$, $s_3 =1$, $s_4 =0$, $s_5 =3$, $s =2$, 
that is,
$\hat E\sim 4 A_{\hat X}$, $M_3\sim A_{\hat X}$, $\hat\MMM_5\subset |3A_{\hat 
X}|$,
$\hat\MMM\subset |2A_{\hat X}|$.
Comparing the dimensions of linear systems we see that
$\hat\MMM=|2A_{\hat X}|$.
The relations~\eqref{eq:-b-gamma-delta-1} have the form
\begin{align*}
4b&=5\delta -7,
\\
4\gamma_3&=\delta-3.
\end{align*}
We obtain $\delta\ge 3$ and $b\ge 2$. 
Therefore, $\bar f(\bar F)$ is a Gorenstein or smooth point   by Lemma~\ref{lemma-discrepancies} and Corollary~\ref{cor:q5:ind2sing}.
\end{proof}

\subsection{Case~\xref{Case:q=7:B=2-2-2-5-8}.}
Then $\B(X)=(2^3, 5, 8)$, $|A_X|=\varnothing$,  $\dim|2A_X|=\dim|3A_X|=0$, and  $\dim|4A_X|=1$.
\begin{sproposition}
\label{prop:q=7:B=2-2-2-5-8}
Let $X$ be a $\QQ$-Fano threefold 
of type~\xref{Case:q=7:B=2-2-2-5-8} in Table~\xref{tab:1}.
Assume that $X$ is not rational. Then there exists a Sarkisov link of the 
form~\eqref{diagram-main}, where $f$ is the Kawamata blowup of 
the point of index~$8$, the variety $\hat X$ is isomorphic to a hypersurface 
$\hat X_{10}\subset\PP(1,2,3,4,5)$, and $\bar f$
contracts a divisor to a Gorenstein or smooth point.
\end{sproposition}

\begin{proof}
In this case the linear system $|4A_X|$ is a pencil without fixed components.
Apply the construction~\eqref{diagram-main} with $\MMM=|4A_X|$.
If $P$ is the point of index~$8$, then $\MMM\overset{P}{\sim} 4(-K_X)$ and so 
$\beta\ge 4\alpha$
by~\eqref{eq:ct:beta-alpha}. The relation~\eqref{eq:main} has the form
\begin{equation}
\label{eq:q=7:B=2-2-2-5-8}
4\hat{q}=7 s+(7\beta-4\alpha) e\ge 7 s+24\alpha e\ge 7 s+3e.
\end{equation}

If $\hat q=1$, then $s=0$, $e=1$, $\alpha=1/8$, and 
$\beta=9/14\notin\frac18\ZZ$, a contradiction.
Thus $\hat q>1$, hence $s\ge 1$ and $\hat q\ge 3$.
If $\alpha\ge 1/2$, then $\hat q\ge 5$.
In this case $s\ge 2$ by Theorem~\ref{thm0}, then $\hat q\ge 6$, $s\ge 4$, 
and
$\hat q>7$, a contradiction. 

Therefore, $\hat q\ge 3$, $f(E)$ is a point of index~$r>1$, and 
$\alpha=1/r$, where $r\in\{5,\, 8\}$.
If $\hat q=3$, then $s=e=1$ and $\beta=(5r+4)/7\notin\frac1r\ZZ$, a 
contradiction.
Therefore, $\hat q\ge 4$ and $s\ge 2$ by Theorem~\ref{thm0}.
In this case $\hat q\ge 5$ by~\eqref{eq:q=7:B=2-2-2-5-8}.
Assume that $\hat q\ge 6$. In this case $s\ge 4$ by Theorem~\ref{thm0} and so 
$\hat q>7$. Again 
we get a contradiction. Therefore, $\hat q=5$ and so $s=2$. 
By Theorem~\ref{thm:q5}
the variety $\hat X$ is isomorphic to a hypersurface $\hat 
X_{10}\subset\PP(1,2,3,4,5)$.
Hence $\Cl(\hat X)=0$ and $e\ge 2$ by Lemma~\ref{lemma:sl:torsion} because 
$|A_X|=\varnothing$.
From~\eqref {eq:q=7:B=2-2-2-5-8} we obtain $e=2$, $\alpha=1/8$, and $r=8$.

From~\eqref{eq:main} 
we also obtain
$s_k =k-2$ for $k=2,\dots,6$,
that is,
$\hat E\sim 2A_{\hat X}$ and
$\hat\MMM_k\subset |(k-2)A_{\hat X}|$ for $k=3,\dots,6$.
Comparing the dimensions of linear systems we see that
$\hat\MMM_4=|2A_{\hat X}|$,
$\hat\MMM_5=|3A_{\hat X}|$, and
$\hat\MMM_6$ is a codimension 1 subsystem in $|4A_{\hat X}|$.
Finally, the relations~\eqref{eq:-b-gamma-delta-1} in our situation have the form
\begin{align*}
2b &=5\delta-7,
\\[2pt]
2\gamma_k &=(k-2)\delta-k\quad\text{for $k=3,\dots,6$}.
\end{align*}
We obtain $\delta\ge 3$ and $b\ge 4$.
Hence, $\bar f(\bar F)$ is a Gorenstein (or smooth) point by Lemma~\ref{lemma-discrepancies} and Corollary~\ref{cor:q5:ind2sing}.
\end{proof}

\begin{scorollary}
$\beta_k<1$ for  $k=2,\dots,6$.
\end{scorollary}

\section{A family of nonrational $\QQ$-Fano threefolds of index~$7$}
\label{sect:q=7}
In this section we consider $\QQ$-Fano threefolds as in  \xref{Case:q=7:B=2-2-2-3-4-5} of Table~\xref{tab:1}
in detail.
Thus we assume that 
\[\textstyle
\qQ(X)=7,\quad  \B(X)=(2^3, 3, 4, 5),\quad   A_X^3=\frac{1}{60}
\]
and then
\[
|A_X|=\varnothing, \quad \dim|2A_X|=\dim|3A_X|=0,\quad\dim|4A_X|= \dim|5A_X|= 1,\quad \dim|6A_X|=2.
\]
The collection of non-Gorenstein singularities of $X$ is as follows:
\begin{itemize}
\item
a unique point $P_5$ of index~$5$ that is a cyclic quotient,
\item
a unique point $P_3$ of index~$3$ that is a cyclic quotient,
\item
a unique point $P_4$ of index~$4$ that is either a cyclic quotient or a
singularity of type~\typeA{cAx}{4},
\item
index~$2$ singularities $P_2^{(1)}$,\dots,$P_2^{(m)}$, where $0\le m\le 3$ and 
$\aw (X,P_4)+\sum\aw(X,P_2^{(i)})=4$.
\end{itemize}

We will show in Theorem~\xref{thm:X14} that a nonrational variety of this type is  isomorphic to a hypersurface $X_{14}\subset\PP(2,3,4,5,7)$.
Moreover, we explicitly describe a birational transformation of this variety  
to a special hypersurface of degree $10$ in $\PP(1,2,3,4,5)$ (see Sect.~\ref{sect:q=5}).
First, we need  auxiliary facts (Lemmas~\ref{lemma:q=7:sl1} and~\ref{lemma:q=7:bl3}) describing two Sarkisov links 
that allow us to make conclusion on some linear systems on $X$ (see Corollaries~\ref{cor:q=7:P_4}, \ref{cor:q=7:P_3}, and~\ref{cor:q=7:P_3-P_4-P_5}). 

\begin{lemma}
\label{lemma:q=7:sl1}
Let $X$ be a $\QQ$-Fano threefold of type~\xref{Case:q=7:B=2-2-2-3-4-5} in Table~\xref{tab:1}. 
Assume that $X$ is not rational. 
Then there exists a Sarkisov link of the form~\eqref{diagram-main}, where $f$ 
is extremal blowup of 
the point of index~$4$, $\bar f$ is birational and  contracts the divisor $\bar 
M_2$
to a point. In this situation, $\hat{X}$ is a $\QQ$-Fano threefold with $\qQ(\hat X)= 3$, $\Clt{\hat X}=0$,
$\p_1(\hat X)\ge 2$,  and $\p_2(\hat X)\ge 3$.
Moreover, we have
\[
s_2 = 0, \quad
s_3 = s_5 = 1,  \quad
e= s_4 =   s_6 = 2,\quad
\text{and}\ \beta_k<1
\ \text{for $k=2,\dots,6$}. 
\]
\end{lemma}

\begin{proof}
Apply the construction~\eqref{diagram-main} with $\MMM=|5A_X|$.
Then $\beta_5\ge 3\alpha$ by~\eqref{eq:ct:beta-alpha}. In particular, $f(E)\neq P_5$ by Corollary~\ref{cor:q=7:P_5}. Taking Theorem~\ref{thm0} into account
we obtain the following possibilities:
\begin{enumerate}
\renewcommand{\theenumi}{\rm (\arabic{enumi})}
\renewcommand{\labelenumi}{\rm (\arabic{enumi})}
\item 
\label{case:q=7=>3:a}
$\alpha= 1/2$, $\hat q= 3$, $ e= 1$, $s_5 = 1$, 

\item 
\label{case:q=7=>3:b}
$\alpha= 1/4$, $\hat q= 3$, $ e= 2$, $ s_5 = 1$, 
\item 
\label{case:q=7=>3:c}
$\alpha= 1/4$, $\hat q= 5$, $e= 1$, $s_5\le 3 $.
\end{enumerate}
Thus $\p_1(\hat X)\ge 2$ if $\hat q= 3$.
Then the group $\Cl(\hat X)$ is torsion free by Corollary~\ref{cor:q=3:tor}. 
Recall that $|A_X|=\varnothing$. Therefore,  $e>1$ by Lemma~\ref{lemma:sl:torsion}  and 
we are in the case~\ref{case:q=7=>3:b}.
Then   from~\eqref{eq:main} we obtain
$s_2 = 0$, $s_3 = 1$, $s_4 =   s_6 = 2$, and   $\beta_k<1$
for $k=2,\dots,6$. In particular, $\beta_4=0$.
Finally, 
\eqref{eq:-b-gamma-delta-1}
has the form
\begin{equation*}
\begin{array}{lll}
2b &=&3\delta-7,
\\[2pt]
2\gamma_5 &=& \delta-5.
\end{array}
\end{equation*}
This implies that $\delta\ge 5$ and $b\ge 4$.
Hence, $\bar f(\bar F)$ is a point.
\end{proof}

\begin{scorollary}
\label{cor:q=7:P_4}
$P_4\notin\Bs|4A_X|$.
\end{scorollary}

\begin{lemma}
\label{lemma:q=7:bl3}
Let $X$ be a $\QQ$-Fano threefold of type~\xref{Case:q=7:B=2-2-2-3-4-5} in Table~\xref{tab:1}. 
Assume that $X$ is not rational. 
Then there exists a Sarkisov link of the form~\eqref{diagram-main}, where $f$ 
is 
the Kawamata blowup of 
the point of index~$3$, $\bar f$ is birational and contracts the divisor $\bar 
M_2$.
Moreover, one of the following holds:
\begin{enumerate}
\item 
\label{lemma:q=7:bl3:a}
$\hat q= 2$, $e= 1$, $\Clt{\hat X}\simeq\ZZ/2\ZZ$;

\item 
\label{lemma:q=7:bl3:b}
$\hat q= 4$, $e= 2$, $\Clt{\hat X}=0$.
\end{enumerate}
In both cases  $s_2=0$ and $s_3=e$.
\end{lemma}

\begin{proof}[Sketch of the proof]
Apply the construction~\eqref{diagram-main} with noncomplete linear system
$\MMM\subset |6A_X|$ consisting of all divisors passing 
through $P_3$. 
Clearly, $\dim\MMM\ge 1$ and $\MMM$ has no fixed components.
Apply the construction~\eqref{diagram-main} with $\MMM$.
Then $\beta\ge 3\alpha$ by~\eqref{eq:ct:beta-alpha}.
By Theorem~\ref{thm0} we have $s > 1$ if $\hat q\ge 4$.
Assume that $e=1$. Then $\Clt{\hat X}\neq 0$ by Lemma~\ref{lemma:sl:torsion} 
because $|A_X|=\varnothing$.
In the case $\hat q=3$ we have $s>1$ by Corollary~\ref{cor:q=3:tor}.
In the case $\hat q=5$ we have $s>2$ by Theorem~\ref{thm:q5}.
Taking these facts into account in addition to~\ref{lemma:q=7:bl3:a} and~\ref{lemma:q=7:bl3:b}  we obtain the following possibilities:
\begin{enumerate}
\renewcommand{\theenumi}{\rm (\arabic{enumi})}
\renewcommand{\labelenumi}{\rm (\arabic{enumi})}

\item
\label{case:q=7:bl:q=5:1/4}
$\alpha= 1/4$,
$\hat q= 5$, $e= 1$, 

\item
\label{case:q=7:bl:q=1}
$\alpha= 1/5$,
$\hat q= 1$, $e= 2$, 

\item
\label{case:q=7:bl:q=4}
$\alpha= 1/5$, $\hat q= 4$, $e= 1$, 

\item
\label{case:q=7:bl:q=5:1/5}
$\alpha= 1/5$,
$\hat q= 5$, $e= s = 3$, $\beta = 3/5$,
\end{enumerate}

In the case~\ref{case:q=7:bl:q=1} we get a 
contradiction by~\eqref{eq:main1} with $k=3$.
In the case~\ref {case:q=7:bl:q=4} by Corollary~\ref{cor:q=7:beta}
we have $\beta_2<1$ and then $s_2>0$ again  by~\eqref{eq:main1}. This implies that 
$|\Clt{\hat X}|\ge 3$
and then $\hat X$ is rational by Corollary~\ref{cor:q=3:tor}, a contradiction.
Similar arguments work in the case~\ref{case:q=7:bl:q=5:1/4} (here we can use Proposition~\ref{prop:t}).
Thus we are left with the 
case~\ref{case:q=7:bl:q=5:1/5} and then $\ct(X,\MMM)=1/3$.
Moreover, both $P_5$ and $P_3$ is a center of canonical singularities for $(X,\frac13\MMM)$.
Then the Kawamata blowup of $P_3$ is crepant with respect to $K_X+\frac13\MMM$, hence 
a Sarkisov link of the form~\eqref{diagram-main} with center $P_3$ 
exists and it must satisfy either~\ref{lemma:q=7:bl3:a} or~\ref{lemma:q=7:bl3:b}.
Finally,  from  \eqref{eq:main1} with $k=2$ and $3$ we obtain $s_2=0$ and $s_3=e$.
\end{proof}

\begin{scorollary}
\label{cor:q=7:P_3}
$P_3\notin M_3$ and $P_3\notin\Bs|6A_X|$.
\end{scorollary}

\begin{proof}
In both cases~\ref{lemma:q=7:bl3:a} and~\ref{lemma:q=7:bl3:b} by \eqref{eq:main1}
we have $\beta_3=0$, that is, $\tilde M_3=f^* M_3$. Since $3M_3\in |6A_X|$,
for the linear system $\tilde \MMM_6=|6A_X|$ we also have $\tilde \MMM_6=f^* \MMM_6$. 
\end{proof}

\begin{scorollary}
\label{cor:q=7:P_3-P_4-P_5}
For general members $M_4\in \MMM_4$ and $M_5\in \MMM_5$ we have
\[
M_2\cap M_3\cap M_4 =\{P_5\}\quad \text{and}\quad M_2\cap M_3\cap M_4 \cap M_5 
=\varnothing.
\]
\end{scorollary}
\begin{proof}
By Corollaries~\ref{cor:q=7:P_4} and~\ref{cor:q=7:P_3} we have $M_2\cap M_3\cap 
M_4\not\ni P_4$
and $M_2\cap M_3\cap M_4\not\ni P_3$. Assume that $M_2\cap M_3\cap M_4$ 
contains a curve, say $C$. Then
$M_5\cdot C\in \frac12 \ZZ$ by Corollary~\ref{cor:q=7:P_5}. On the other hand, 
$M_2\cdot M_3\cdot M_5=1/2$, hence $C=M_2\cap M_3\ni P_4$, a contradiction.
Therefore, the set $M_2\cap M_3\cap M_4$ is zero-dimensional. Since $M_2\cdot 
M_3\cdot M_4=2/5<1/2$, this set does not contain
points of index~$\le 2$, hence $ M_2\cap M_3\cap M_4 =\{P_5\}$.
\end{proof}

\begin{proposition}
\label{prop:q=7m}
In the notation of Proposition~\xref{prop:q=7} the map $\chi$ is an 
isomorphism, i.e. 
the link~\eqref{diagram-main} has the form 
\begin{equation}
\label{eq:sl:q=7}
\vcenter{
\xymatrix@C=3em{
&\tilde{X}\ar@/_0.4em/[dl]_{f}\ar@/^0.4em/[dr]^{\bar{f}}
\\
X\ar@{-->}[rr]^{g}&&\hat{X}
} }
\end{equation} 
and $\bar f$ contracts
a divisor $\tilde M_3$ to a smooth rational curve $\hat\Upsilon\subset\hat X$
such that $A_{\hat X}\cdot\hat\Upsilon=1/2$.
The image of the $f$-exceptional divisor $E\simeq \PP(1,1,4)$ is a member of $|3A_{\hat X}|$ that is singular along $\hat\Upsilon$.
The non-Gorenstein singularities of $\hat X$ are as follows:
\begin{itemize}
\item
a cyclic quotient singularity $\hat Q_4$ of index~$4$;
\item
a cyclic quotient singularity $\hat P_3=g( P_3)$ of index~$3$;
\item
an index-$2$ singularity $\hat P_4=g( P_4)$ that is not a cyclic quotient.
\end{itemize}
In particular, $\hat X$ is not quasi-smooth. 
\end{proposition}

\begin{proof}
We use the notation of the proof of Proposition~\xref{prop:q=7}.
First, we claim that $\dim (M_3\cap\Bs |4A|)\le 0$.
Indeed, assume that there exists an irreducible curve $\Gamma$ contained in $M_3\cap\Bs |4A|$.
By Corollaries~\ref{cor:q=7:P_4} and~\ref{cor:q=7:P_3} we have $P_3,\, 
P_4\notin \Gamma$.
Let $M_5\in \MMM_5$ be a general member. 
By Corollary~\ref{cor:q=7:P_5} \ $P_5\notin M_5$.
Therefore, $M_5\cdot \Gamma\in \frac 12 \ZZ$.
On the other hand, $\Gamma\subset M_2\cap M_3$ and so $0<M_5\cdot \Gamma\le M_5 
\cdot M_2\cdot M_3=1/2$.
Hence $\Gamma= M_2\cap M_3\ni P_4$, a contradiction.
Thus, $\dim (M_3\cap\Bs |4A|)\le 0$.

Now we claim that $\tilde\MMM_{4}$ is nef.
Assume that $\tilde\MMM_{4}\cdot \tilde{\Gamma}<0$ for some irreducible curve $\tilde{\Gamma}$.
Since $\tilde\MMM_{4}$ has no fixed components and $\uprho(E)=1$, we have 
$E\cdot \tilde{\Gamma}\ge 0$ and $\tilde{\Gamma}\notin E$.
On the other hand, it follows from~\eqref{eq:main1} that $\beta_3 = 4/5$ and 
$\beta_4 = 2\beta_2 = 2/5$.
Therefore,
\begin{equation*}
%\label{eq:q=7:M2M3}
\textstyle
\tilde\MMM_4\qq f^*(4A_X)-\frac 25 E\qq 2\tilde M_2\qq\frac 43\tilde{F}+\frac 
23 E.
\end{equation*}
Hence $\tilde{F}\cdot \tilde{\Gamma}<0$. Thus the curve $\Gamma:=f(\tilde{\Gamma})$ is contained in $\Bs |4 A_X|$ 
and in the surface $M_3$.
In other words, $\Gamma$ is contained in the set $\{x_2=x_4=x_3=0\}\cap X$ that 
is zero-dimensional,
a contradiction. This proves that $\tilde\MMM_{4}$ is nef.

Since $\gamma_4=0$, we have $\bar\MMM_4\qq\bar f^*\hat\MMM_4$.
In particular, $\bar\MMM_4$ is nef. Since the varieties $\tilde X$ and $\bar X$ 
are isomorphic in codimension one, we see that for $n\gg 0$ the linear systems 
$|n\tilde\MMM_{4}|$ and $|n\bar\MMM_4|$ define morphisms 
$\Phi_{|n\tilde\MMM_{4}|}:\tilde X\to\PP^N$
and $\Phi_{|n\bar\MMM_{4}|}:\bar X\to\PP^N$ with the same image
that must coincide with~$\hat X$.
This implies that $\Phi_{|n\tilde\MMM_{4}|}$ contracts a divisor and so $\chi$
is an isomorphism. Let $\hat \Upsilon:=\bar f(\tilde M_3)$. Since 
$\OOO_E(E)=\OOO_{\PP(1,1,4)}(5)$, we have $E^3=25/4$. 
Taking this fact into account, we obtain
\[
\textstyle
2A_{\hat X} \cdot \hat \Upsilon=- \tilde\MMM_{4}\cdot (\tilde 
M_3)^2=-\left(f^*(4A_X)-\frac 25 E\right)\cdot \left(f^*(3A_X)-\frac 45 
E\right)^2=1.
\]

The index-$5$ point of $X$ is of type $\frac15(1,1,4)$. Hence the 
non-Gorenstein singularities of $\tilde X$ are as follows:
\begin{itemize}
\item
a cyclic quotient singularity $\tilde Q_4$ of index~$4$, where $\tilde Q_4\in 
E$;
\item
index-$2$ singularities $\tilde P_2^{(i)}=f^{-1}( P_2^{(i)})$, where $\tilde 
P_2^{(i)}\in \tilde M_3$, $\tilde P_2^{(i)}\notin E$;
\item
an index-$4$ singularity $\tilde P_4=f^{-1}( P_4)$, where $\tilde P_4\in \tilde 
M_3$, $\tilde P_4\notin E$;
\item
a cyclic quotient singularity $\tilde P_3=f^{-1}( P_3)$ of index~$3$, where 
$\tilde P_3\notin \tilde M_3$, $\tilde P_3\notin E$.
\end{itemize}

Then $\hat P_3:=\bar f(\tilde P_3)$ is a cyclic quotient singularity of index~$3$, 
$\bar f(\tilde P_4)$ is a point of index~$<4$, and $\bar f(\tilde P_2^{(i}))$ 
are Gorenstein points.
Since $\B(\hat X)=(2^2,3,4)$, we see that $\hat Q_4:=\bar f(\tilde Q_4)$ must 
be a point of index~$4$, hence 
$\bar f$ is an isomorphism near $\tilde Q_4$. Hence $\tilde{Q_4}\notin \tilde{M_3}$ and 
$\hat{Q_4}\notin \hat{\Upsilon}$. Then $\bar f(\tilde P_4)$ must be 
a point of index~$2$ that is a not a cyclic quotient by \cite{Kawamata:Div-contr}.
The curve $\hat \Upsilon$ does not pass through $\hat P_3$ nor through $\hat 
Q_4$. Hence the divisor $2A_{\hat X}$ is Cartier near $\hat \Upsilon$.
It can be seen from \eqref{eq:thm:q5-1:a} and \eqref{eq:thm:q5-1:b} that the 
set $\hat E\cap \Bs |2A_{\hat X}|$ is given by 
$\{x_1=x_2=x_3=0\}$, hence $\hat E\cap\Bs |2A_{\hat X}|=\{\hat Q_4\}$.
Since $\hat Q_4\notin \hat \Upsilon$, the restriction of $|2A_{\hat X}|$ to 
$\hat\Upsilon$
is a linear system of positive dimension (and degree $1$). This implies that 
$\hat \Upsilon$ is a smooth rational curve. 
\end{proof}

\begin{theorem}
\label{thm:X14}
Let $X$ be a $\QQ$-Fano threefold with $\qQ(X)=7$, $\B(X)=(2^3, 3, 4, 5)$, and 
$A_X^3=1/60$.
Assume that $X$ is not rational.
Then $X$ is isomorphic to a hypersurface $X_{14}\subset\PP(2,3,4,5,7)$.
\end{theorem}

\begin{proof}
For $m=3,\dots,7$, let $\varsigma_m$ be a general element of $ H^0 (X, mA_X )$.
By Corollary \ref{cor:q=7:P_3-P_4-P_5} the map
\[
\Psi: X \dashrightarrow \PP(2,3,4,5),\qquad P \longmapsto 
(\varsigma_2(P),\varsigma_3(P),\varsigma_4(P),\varsigma_5(P))
\]
is a morphism. For short, let $\PP:=\PP(2,3,4,5)$ and let $A_{\PP}$ be the 
positive generator of $\Cl(\PP)$.
By the construction, $A_X=\Psi^* A_{\PP}$.
Since $A_X^3=1/60$
and $A_{\PP}^3=1/120$, the morphism $\Psi$ is finite of degree $2$.
By the Hurwitz formula we have
\[
\textstyle
\Psi^*(7 A_{\PP})=7A_X = -K_X=\Psi^* \left(-K_{\PP} +\frac 12 R\right)
=\Psi^* \left(14 A_{\PP} -\frac 12 R\right),
\]
where $R$ is the branch divisor. This gives us $R\sim 14 A_{\PP}$. Therefore, 
$X$ is a hypersurface of degree~$14$ in $\PP(2,3,4,5,7)$.
\end{proof}
\newcommand{\etalchar}[1]{$^{#1}$}

% \bibliography{all,prokho}
% \bibliographystyle{alpha}
\end{document}